\documentclass[10pt,a4paper]{article}
\usepackage{latexsym}
\usepackage{amsthm,amsmath,amssymb}
\usepackage[english]{babel}
\usepackage{enumerate}
\usepackage{ dsfont}
\usepackage{graphicx}
\usepackage{color}
\usepackage{fullpage}
\usepackage{subfigure}
\usepackage{array}

\def\ep{{\varepsilon}}
\def\R{\mathbb R}
\def\rac{\sqrt{1+B^2}}

\newtheorem{theo}{\textbf{Theorem}}[section]

\newtheorem{lem}[theo]{\textbf{Lemma}}
\newtheorem{prop}[theo]{\textbf{Proposition}}

\newtheorem{rem}[theo]{\textbf{Remark}}

\newtheorem{ass}[theo]{\textbf{Assumption}}

\def\signma{\bigskip \begin{center} {\sc Matthieu Alfaro\par\vspace{3mm}
      Universit\'e Montpellier\par
      IMAG,CC051, Place Eug\`ene Bataillon, 34095 Montpellier Cedex 5, France
      \par\vspace{3mm} e-mail:}
    \tt{matthieu.alfaro@umontpellier.fr} \end{center}}

\def\signhb{\bigskip \begin{center} {\sc Henri Berestycki\par\vspace{3mm}
      EHESS, PSL Research University, CNRS,\par
      Centre d'analyse et mathématique sociales, Paris, 190-198 avenue de France,
75013 - Paris, France 
      \par\vspace{3mm} e-mail:}
    \tt{hb@ehess.fr} \end{center}}

\def\signgr{\bigskip \begin{center} {\sc Ga\"el
      Raoul\par\vspace{3mm}
      CMAP, Ecole Polytechnique, CNRS, Université Paris-Saclay, Route de Saclay, 91128 Palaiseau cedex, France.\par
      \par\vspace{3mm} e-mail:}
    \tt{raoul@cmap.polytechnique.fr} \end{center}}

\title{The effect of climate shift on a species submitted to dispersion, evolution, growth and nonlocal competition}

\begin{document}

\maketitle

\begin{center}
\author{M. Alfaro, H. Berestycki,  G. Raoul}
\end{center}

%
%

\tableofcontents

\vspace{10pt}

\begin{abstract} We consider a population structured by a space
variable and a phenotypical trait, submitted to dispersion,
mutations, growth and nonlocal competition. We introduce the
climate shift due to {\it Global Warming} and discuss the dynamics
of the population by studying the long time behavior of the
solution of the Cauchy problem. We consider three sets of
assumptions on the growth function. In the so-called {\it confined
case} we determine a critical climate change speed for the
extinction or survival of the population, the latter case taking place by \lq\lq strictly following the climate shift''. In the so-called {\it
environmental gradient case}, or {\it unconfined case}, we additionally determine the propagation speed
of the population when it survives: thanks to a combination of migration and evolution, it can here be different from the speed of the climate shift. Finally, we consider {\it mixed scenarios}, that are complex situations, where the
growth function satisfies the conditions of the confined case on the right, and the conditions of the unconfined case on the left.

The main difficulty comes from the nonlocal competition term that prevents
 the use of classical methods based on comparison arguments. This difficulty
  is overcome thanks to estimates on the tails of the solution, and a careful  application of the parabolic Harnack inequality.\\


\end{abstract}

\section{Introduction}\label{s:intro}

We consider a density of population $n(t,x,y)$ at time $t\geq 0$,
structured by a spatial variable $x\in \mathbb R$ and a phenotypic
trait $y\in \mathbb R$.  The population  is  submitted to four
essential processes: spatial dispersion, mutations, growth and
competition. It also has to face the climate shift induced by {\it
Global Warming}. The spatial dispersion and the mutations are
modelized by diffusion operators. We assume that the growth rate
of the population, if the competition is neglected, depends
initially ($t=0$) on both the location $x$ and the phenotypic trait
$y$. Then, in order to take into account the climate shift, we
assume that for later times ($t>0$) the conditions are shifted in
space, at a given and forced speed $c$ (without loss of
generality, we will always consider $c\geq 0$). Hence, the growth
rate is given by $r(x-ct,y)$, which is typically negative outside
a strip centered on the line $y-B(x-ct)=0$ (see the {\it
environmental gradient case} below, and notice that we also study other situations). This corresponds to a population living
in an environmental cline: to survive at time $t$ and location $x$,
an individual must have a trait close to the optimal trait
$y_{opt}(t,x)=B(x-ct)$ which is shifted by the climate (one may
think of $x$ being the latitude). Finally, we consider a logistic
regulation of the population density that is local in the spatial
variable  and nonlocal in the trait. In other words,  we consider
that there exists an intra-specific competition (for e.g. food) at
each location, which may depend on the traits of the competitors.
The model under consideration is then (after a rescaling in $t$,
$x$ and $y$)
\begin{equation}\label{eq_n}
 \left\{\begin{array}{lll}
          \partial_t n(t,x,y)-\partial_{xx} n(t,x,y)-\partial_{yy} n(t,x,y)&&\vspace{3pt}\\
      \quad\quad =\displaystyle \left(r(x-ct,y)-\int_{\mathbb
R}K(t,x,y,y')n(t,x,y')\,dy'\right)n(t,x,y) &&\textrm{ for }(t,x,y)\in\R_+\times \R^2,\vspace{5pt}\\
      n(0,x,y)=n_0(x,y) &&\textrm{ for }(x,y)\in \R^2,
         \end{array}
\right.
\end{equation}
with precise assumptions to be stated later.

In this paper our aim is to determine conditions that imply extinction of the population and the
ones that imply its survival, or even its propagation. Typically,
we expect the existence of a
critical value $c^*>0$ for the forced speed $c$ of the climate
shift: the population goes extinct (in the sense that it can not
adapt or migrate fast enough to survive the climate shift) when $c\geq c^*$ and
survives, by following the climate shift and/or thanks to an adaptation of the individuals to the changing climate, when $0\leq c<c^*$. 
To confirm these scenarios, we shall study the long time behavior of
a global nonnegative solution $n(t,x,y)$ of the Cauchy problem
\eqref{eq_n}.

The model \eqref{eq_n} can be seen as a reaction-diffusion equation with a monostable reaction term. Solutions of such equation typically propagate in space at a linear speed, that can often be explicitly determined. In some models, such as the Fisher-KPP equation \cite{Fisher}, \cite{KPP}, it is actually possible to push the analysis beyond the propagation speed: one can for instance describe the convergence of the population to a travelling wave, see \cite{Bramson1,Bramson2}. The analysis of \eqref{eq_n} is however much more involved, in particular because of the nonlocal competition term it contains. We will then focus our analysis on the qualitative properties of the solutions, based on the notion of spreading introduced in \cite{Aronson}.

%
%

In this paper we investigate three main types of problems giving rise to qualitatively different behaviors.
These correspond to different assumptions about the region where the growth function $r$ is positive. We now define these various cases and state some of the main results we obtain for each one of them.

\subsection{The confined case}

In the {\it confined case}, the growth function can be positive only in a bounded (favorable) region of the $(x,y)$ plane. The precise assumption is as follows.

\begin{ass}[Confined case]\label{hyp-confined}
For all $\delta >0$, there is $R>0$ such that
\begin{equation}\nonumber
r(x,y)\leq -\delta \; \text{ for almost all  $(x,y)$ such that
$|x|+|y|\geq R$}.
\end{equation}
\end{ass}

The main results in this case are given in Proposition \ref{prop:extinction},
Proposition~\ref{prop:extinction2} and in Theorem~\ref{th:survival}. Essentially these state that there is a {\em critical speed} $c^*$ such that when the climate change speed $c$ is such that $c<c^*$, then the population persists by keeping pace with the climate change, whereas when $c>c^*$, the population becomes extinct as time goes to infinity. 
We obtain the critical speed from an explicit (generalized) eigenvalue problem. We explain these notions in  subsection~\ref{subsection:eigenvalue_pb}.

\subsection{The environmental gradient case}

In ecology, an {\em environmental gradient} refers to a gradual change in various factors in space that determine the favored phenotypic traits. Environmental gradients can be related to factors such as altitude, temperature, and other environment characteristics. In our framework this case is defined by the following condition.

\begin{ass}[Environmental gradient case, or unconfined case]\label{hyp-unconfined} We assume that $r(x,y)=\bar
r(y-Bx)$ for some $B>0$ and some function $\bar r \in L^\infty
_{loc}(\R)$ such that, for all $\delta >0$, there is a $R>0$ such
that
\begin{equation}\nonumber
\bar r(z)\leq -\delta \; \text{ for almost all $z$ such that
$|z|\geq R$.}
\end{equation}
\end{ass}

Note that in this case, without climate change, the favorable region, where $r(x,y)$ can be positive, spans an unbounded slab, in contradistinction with the confined case where it is bounded.

The main results in this case are given in
Proposition~\ref{prop:extinction-unconfined} and Theorem~\ref{th:invasion}. As in the confined case, there is a critical speed $c^{**}$ for the climate shift --- obtained from a generalized eigenvalue problem, and even explicitly computable in some simple
situations, see formula \eqref{formula-critical-speed}--- which separates extinction from survival/invasion. Nevertheless, let us emphasize on a main difference with the confined case:  the population does not necessarily keep pace 
with  the climate but  may persist thanks to a combination of migration and evolution, see Remark~\ref{rem:migration-evolution}. To shed light on this phenomenon, in Theorem~\ref{th:invasion} we further  identify the propagation speed of a population in an environmental gradient when $c<c^{**}$, which we believe to be important for further investigations. 

\medskip

Before going further, let us present simple but meaningful examples of such growth functions:
\begin{equation}\label{r-quadratic}
 r_\ep(x,y):=1-A(y-Bx)^2-\ep x^2,
\end{equation}
for some $A>0$, $B>0$, $\ep \geq 0$. At every spatial location
$x\in\mathbb R$, the optimal phenotypic trait (i.e. the phenotypic
trait that provides the highest growth rate) is $y_{opt}=Bx$, and
the constant $B$ represents the linear variation on this optimal
trait in space. The constant $A$ characterizes the quadratic
decrease of the growth rate $r$ away from the optimal phenotypic
trait. Finally, the constant $\ep$ describes how the optimal
growth rate varies in space: if $\ep>0$,  we assume that an individual originating from a given region can adapt to warmer temperatures induced by the climate shift, but will not, nevertheless, be as successful as it
was originally (in the sense that its growth decreases).
  When $\ep>0$, $r_\ep$ given by \eqref{r-quadratic}
satisfies Assumption \ref{hyp-confined} and the population, if it
survives, cannot invade the whole $(x,y)$ plane (confined case).
This situation will be discussed in Section \ref{s:survival}. On the other
hand, for $\ep =0$, $r_0$ given by \eqref{r-quadratic} satisfies
Assumption \ref{hyp-unconfined} and the possibility of propagation
remains open (unconfined case). This situation will be discussed in Section
\ref{s:propagation}.

\subsection{The mixed case}
Finally, we introduce here a more complex situation, that combines 
the two previous cases. We call it the {\it mixed case}.  It is defined by the following condition.

\begin{ass}[Mixed case]\label{hyp-mixed} We have
$$
r(x,y)=\mathbf 1_{\R_-\times \R}(x,y)r_u(x,y)+\mathbf
1_{\R_+\times \R}(x,y)r_c(x,y),
$$
 where $r_c$ satisfies Assumption \ref{hyp-confined} and $r_u$
 satisfies Assumption \ref{hyp-unconfined}.
\end{ass}

Thus, in this case, the growth
function satisfies the assumption of the environmental gradient case
 for $x\leq 0$, and the assumption of the confined case for $x\geq 0$. A typical example is given by
\begin{equation}\label{example-mixed}
 r(x,y)=1-A(y-B{x_-}-B'{x_+})^2-\ep {x_+}^2,
\end{equation}
for some $A>0$, $B>0$, $B'>0$ and $\ep>0$. Note that
we have used the notations $x_-=\min(x,0)$, $x_+=\max(x,0)$.

\medskip

Our main results in this case are given in subsection~\ref{subsection:theomixed}. 
We show that one can extend the critical speeds, $c^*$ and $c^{**}$, obtained in the  previous two cases.
We prove that the population persists if the climate change speed is below either one of these critical speeds whereas it goes extinct when it is above both. This is stated in details  in Theorem~\ref{th:mixed} where we also describe more precisely the large time behavior of the population, depending on the position of $c$ w.r.t. $c^{*}$ and $c^{**}$. Furthermore, a new and interesting phenomenon arises in this case: when $c^*$ and $c^{**}$ are close, then, as $c$ varies, the dynamics of the population can rapidly change from fast expansion to extinction, see Remark~\ref{remint}.

\subsection{Estimating the nonlocal competition term}\label{ss:estimating}

When studying the extinction cases, it follows from the parabolic
comparison principle that we can neglect the nonlocal term
$\int_{\mathbb R}K(t,x,y,y')n(t,x,y')\,dy'$ in \eqref{eq_n}. On
the contrary, careful estimates on this nonlocal term are
necessary to study the survival and propagation phenomena. The
strategy consists in first proving estimates on the tails of the
solutions, which provides a control of the nonlocal term for large
$x$, $y$. Next, on the remaining compact region, a rough uniform
bound on the mass $\int _\R n(t,x,y)\,dy$ enables us to apply an argument based on the 
parabolic Harnack inequality for linear equations with bounded
coefficients, and therefore to control the nonlocal term. This idea is similar to the method developed in \cite{Alf-Cov-Rao1} to study travelling wave solutions of a related problem. In this previous work however, the solution was time independent, and we could use the elliptic Harnack inequality. In the present work, an additional difficulty arises: for parabolic equations, the Harnack inequality involves a necessary time shift, see Remark~\ref{Rk:Harnack}. Nevertheless, we show  in subsection~\ref{subsec:harnack} that if the solution $u$ of a parabolic Harnack inequality is uniformly bounded (which, in our situations, will be proved in Lemma~\ref{lem:tails} and Lemma \ref{lem:tails2}), then for any $\bar t>0$, $\bar x \in \R ^N$, $R>0$ and $\delta>0$, there exists $C>0$ such that the solution $u$  satisfies
\[\max_{x\in B(\bar x,R)}u(\bar t, x)\leq C\min_{x\in B(\bar x,R)}u(\bar t, x)+\delta,
\]
thus getting rid of the time shift. This refinement of the parabolic Harnack inequality is a very efficient tool for our analysis, since used in the proofs of  Lemma \ref{lem:tails} (exponential decay of
tails), Theorem \ref{th:survival} (survival in the confined
case), Theorem \ref{th:invasion} (survival and invasion in the unconfined case), Theorem \ref{th:mixed} (iii) (survival and
invasion in the mixed case). Since we believe that such rather involved technics are also of independent interest for further utilizations, we present them as a separate result in Theorem~\ref{th:harnack-mod} for a  general parabolic equation.
\subsection{Related works and comments}
In
\cite{Ber-Die-Nag-Zeg}, a model has been introduced to study the
effect of {\it Global Warming} on a species, when evolutionary
phenomena are neglected, that is when all individuals are assumed
to be  identical (see also \cite{PL}, \cite{Ber-Ros2, Ber-Ros3}).
Among more detailed results, it is shown that there exists a
critical speed $c^*$ such that the population survives if and only
if the climate change occurs at a speed slower than $c^*$.
However, it is well documented that species adapt to local
conditions, see e.g. \cite{SPK}, and in particular to the local
temperatures. Two closely related models taking into account this
heterogeneity of the population have been proposed in \cite{Pease}
and \cite{KB} (see e.g. \cite{PBM}, \cite{MR}, \cite{DMCKR} for
recent results). The models of \cite{Pease}, \cite{KB} describe
the evolution of the population size and its mean phenotypic
trait, and can be derived formally from a structured population
model similar to \eqref{eq_n}, provided the population reproduces
sexually (see \cite{MR}). Such simplified models do not exist for
asexual populations, so that one has to consider \eqref{eq_n} in
this latter case. In this framework, let us mention the
construction of travelling waves \cite{Alf-Cov-Rao1} for equation
\eqref{eq_n}, and \cite{ber-jin-sil} for a related but different
model, when there is no climate shift ($c=0$). Notice also that
the model \eqref{eq_n} can be derived as a limit of stochastic
models of finite populations \cite{CM}.

 The well-posedness of a Cauchy problem
very similar to \eqref{eq_n}, but on a bounded domain, has been
studied in \cite[Theorem I.1]{Prevost}, under reasonable
assumptions on the coefficients. We believe a similar argument
could be used here to show the existence of a unique solution
$n_R=n_R(t,x,y)$, for $(x,y)\in[-R,R]^2$. The existence and
uniqueness of solutions to \eqref{eq_n} could then be obtained
through a limit $R\to\infty$, thanks to the estimates on the tails
of the solutions obtained in Lemma~\ref{lem:tails}. We have
however chosen to focus on the qualitative properties of the
solutions in this article.

The main difficulty in the mathematical analysis of \eqref{eq_n}
is to handle the nonlocal competition term. When the competition
term is replaced by a local (in $x$ and $y$) density regulation,
many techniques based on the comparison principle --- such as some
monotone iterative schemes or the sliding method --- can be used
to get, among other things, monotonicity properties of the
solution. Since integro-differential equations with a nonlocal
competition term do not satisfy the comparison principle, it is
unlikely that such techniques apply here. Problem \eqref{eq_n}
shares this difficulty with the nonlocal Fisher-KPP equation
\begin{equation}\label{nonlocal-kpp}
\partial _t n (t,x)-\partial _{xx}n(t,x)=\left(1-\int _{\R} \phi(x-y)n(t,y)\,dy\right)n(t,x),
\end{equation}
which describes a population structured by a spatial variable
only,  and submitted to nonlocal competition modelized by the
kernel $\phi$. As far as equation \eqref{nonlocal-kpp} is
concerned, let us mention the possible destabilization of the
steady state $u\equiv 1$ by some kernels \cite{GVA}, the
construction of travelling waves \cite{Ber-Nad-Per-Ryz},
additional properties of these waves \cite{Fan-Zha},
\cite{Alf-Cov}, and a spreading speed result \cite{HR}. We also
refer to \cite{A-Cov-Rao2}, \cite{Gri-Rao} for the construction of travelling
waves for a bistable nonlocal equation, for an epidemiological  system with mutations respectively. 
\medskip

\subsection{Mathematical assumptions, and organization of the paper}

Throughout the paper we always assume the following on the coefficients of the nonlocal reaction diffusion equation \eqref{eq_n}: $r\in
L^\infty_{loc}(\mathbb R ^2)$ and there exists $r_{max}>0$ such
that
\begin{equation}\label{Assumption_r}
 r(x,y)\leq r_{max}\; \text{ a.e. in } \R ^2;
\end{equation}
also $K\in L^\infty((0,\infty)\times \R ^3)$ is bounded from above
and from below, in the sense that there are $k^->0$, $k^+>0$ such
that
\begin{equation}\label{Assumption_K}
 k^-\leq K(t,x,y,y')\leq  k^+ \;\text{ a.e. in }  (0,\infty)\times
\R ^3.
\end{equation}

Moreover, we consider initial conditions $n_0(x,y)$ for which there
exists $C_0>0$ and $\mu _0>0$ such that
\begin{equation}\label{initial-data}
0\leq n_0(x,y)  \begin{cases}\leq C_0 e^{-\mu_0 (|x|+|y|)} &\text{
under Assumption \ref{hyp-confined}}\vspace{3pt}
\\
\leq C_0 e^{-\mu_0|y-Bx|} &\text{ under Assumption
\ref{hyp-unconfined}}\\
\text{is compactly supported} &\text{ under Assumption
\ref{hyp-mixed}}.
\end{cases}
\end{equation}
In other words, under Assumptions \ref{hyp-confined} or \ref{hyp-unconfined}, we
allow the initial data to have tails which are \lq\lq consistent''
with the case under consideration. In the mixed case, i.e.
Assumption \ref{hyp-mixed}, for the sake of simplicity, we assume
that the initial data is compactly supported.

\medskip

The organization of this work is as follows. In Section
\ref{s:tails} we provide some linear material (principal
eigenvalue, principal eigenfunction), a preliminary estimate of
the tails of $n(t,x,y)$ together with an efficient Harnack tool which is also
of independent interest. The confined case is studied in Section
\ref{s:survival}: we identify the critical speed $c^*$ and,
depending on $c$, prove extinction or survival. The unconfined
case is studied in Section \ref{s:propagation}: we identify the
critical speed $c^{**}$ and, depending on $c$, prove extinction or
propagation. Finally, Section \ref{s:mixed} is devoted to the
analysis of the mixed case, for which we take advantage of the analysis of the confined and unconfined cases, performed in the two previous sections.

\section{Preliminary results}\label{s:tails}

Let us first introduce a principal eigenvalue problem that will be
crucial in the course of the paper. It will in particular provide
the critical climate shift speeds $c^*$, $c^{**}$, $c^{**}_u$ (see
further) that will allow the survival of the population.

\subsection{A principal eigenvalue problem}\label{subsection:eigenvalue_pb}

The theory of generalized principal eigenvalue developed in
\cite{ber-nir-var} is well adapted to the present problem,
provided $r$ is bounded. Following \cite{ber-nir-var}, we can then
define, for $r\in L^\infty(\Omega)$ and $\Omega\subset \mathbb
R^2$ not necessarily bounded, the generalized principal eigenvalue
\begin{equation}\label{def:lambda}
 \lambda(r,\Omega):=\sup\left\{\lambda:\,\exists \phi\in W^{2,2}_{loc}(\Omega),\,\phi>0,\, \partial_{xx}\phi(x,y)+\partial_{yy}\phi(x,y)+
 \left(r(x,y)+\lambda\right)\phi(x,y)\leq 0\right\}.
\end{equation}
As shown in \cite{ber-nir-var}, if $\Omega$ is bounded and smooth,
$\lambda(r,\Omega)$ coincides with the Dirichlet principal
eigenvalue $H(r,\Omega)$, that is the unique real number such that
there exists $\phi>0$ on $\Omega$ (unique up to multiplication by a scalar),
$$\left\{\begin{array}{ll}
          -\partial_{xx}\phi(x,y)-\partial_{yy}\phi(x,y)-r(x,y)\phi(x,y)
=H(r,\Omega)\phi&\textrm{ a.e. in }\Omega,\\
\phi=0&\textrm{ on }\partial \Omega.
         \end{array}
\right.
$$
Notice that since the operator is self-adjoint, the Dirichlet principal eigenvalue can be obtained through the variational formulation
$$
H(r,\Omega)=\inf_{\phi\in H^1_0(\Omega),\;\phi\neq 0}\frac{\int_\Omega |\nabla \phi|^2-r(x,y)\phi^2}{\int_\Omega \phi^2}.
$$

The following proposition then provides known properties of
$\lambda(r,\Omega)$. We refer the reader to \cite{ber-nir-var},
\cite[Proposition 4.2]{ber-ham-ros}, or to \cite[Proposition
1]{Ber-Ros2} for more details and proofs.

\begin{prop}[Generalized eigenvalues and eigenfunctions]\label{prop:lambda}
Assume that $r\in L^\infty(\mathbb R^2)$. There is $\lambda(r,\Omega)\in\mathbb R$ such that for any subsequence $(\Omega_n)_{n\in\mathbb N}$ of non empty open sets
such that
$$\Omega_n\subset \Omega_{n+1},\quad \cup_{n\in\mathbb N}\Omega_n=\Omega.$$
Then, $\lambda(r,\Omega_n)\searrow \lambda(r,\Omega)$ as
$n\to\infty$. Furthermore, there exists a generalized principal
eigenfunction, that is a positive function $\Gamma \in
W^{2,2}_{loc}(\R ^2)$ such that
\begin{equation*}
-\partial_{xx}\Gamma(x,y)-\partial_{yy}\Gamma(x,y)-r(x,y)\Gamma(x,y)
=\lambda(r,\Omega)\Gamma(x,y)\quad \textrm{a.e. in }\Omega.
\end{equation*}
\end{prop}

Let us also mention that the above generalized eigenfunction $\Gamma$ is indeed obtained as a limit of principal eigenfunctions, with Dirichlet
boundary conditions, on increasing bounded domains.

Since our growth functions are only assumed to be bounded from
above, we extend definition \eqref{def:lambda} in a natural way to $r\in
L^\infty_{loc}(\Omega)$ such that $r\leq r_{max}$ on $\Omega$, for
some $r_{max}>0$. The set
$$\Lambda(r,\Omega):=\left\{\lambda:\,\exists \phi\in W^{2,2}_{loc}(\Omega),\,\phi>0,
 \,\partial_{xx}\phi(x,y)+\partial_{yy}\phi(x,y)+\left(r(x,y)+\lambda\right)\phi(x,y)\leq 0\right\}$$
is not empty, since
$\Lambda(\max(r,-M),\Omega)\subset\Lambda(r,\Omega)$, and is
bounded from above, thanks to the monotony property of
$\Omega\mapsto\Lambda(r,\Omega)$. Finally, going back for example
to the proof of  \cite[Proposition 4.2]{ber-ham-ros}, we notice that
Proposition \ref{prop:lambda} remains valid under the weaker
assumption that $r\in L^\infty_{loc}(\Omega)$ is bounded from
above.

\medskip

It follows from the above discussion that, in the confined case,
we are equipped with the generalized principal eigenvalue $\lambda
_\infty \in \R$, and a generalized eigenfunction  $\Gamma
_\infty(x,y)$ such that $\Gamma_\infty(x)\to 0$ as $|x|\to\infty$ and
\begin{equation}\label{vp-pb}
\begin{cases}
-\partial_{xx}
\Gamma_\infty(x,y)-\partial_{yy}\Gamma_\infty(x,y)-r(x,y)\Gamma_\infty(x,y)=
\lambda_\infty\Gamma_\infty(x,y)  \quad\text{ for all  } (x,y)\in \R^2\\
\Gamma_\infty(x,y) >0\quad \text{ for all } (x,y)\in \R^2,\quad
\Vert \Gamma _\infty \Vert _{L^\infty(\R^2)}=1.
 \end{cases}
 \end{equation}
 Notice  that since the operator is self-adjoint and the potential \lq\lq confining'', $\lambda _\infty$ can also be obtained through some adequate variational formulation.
 
For the unconfined case, for which $r(x,y)=\bar r(y-Bx)$, we are
equipped with the \lq\lq one dimensional'' generalized principal eigenvalue
$\lambda _\infty \in \R$, and a generalized eigenfunction  $\Gamma
_\infty^{1D}(z)$ such that
\begin{equation}\label{vp-pb2}
\begin{cases}
-(1+B^2)\partial_{zz}\Gamma_\infty^{1D}(z)-\bar r(z)\Gamma_\infty
^{1D}(z)=
\lambda_\infty\Gamma_\infty^{1D}(z)  \quad\text{ for all  } z\in \R\\
\Gamma_\infty^{1D}(z) >0\quad \text{ for all } z\in \R,\quad \Vert
\Gamma _\infty^{1D} \Vert _{L^\infty(\R)}=1.
 \end{cases}
 \end{equation}
 Indeed, this corresponds to the confined case in 1D.  In order to be consistent, if we define
 $\Gamma _\infty (x,y):=\Gamma _\infty ^{1D}(y-Bx)$ then
 \eqref{vp-pb} remains valid.

\begin{rem} When the unconfined growth rate is the prototype example \eqref{r-quadratic}
with $\ep=0$, \eqref{vp-pb2} corresponds to the harmonic oscillator, for which the the principal eigenvalue and principal eigenvector are known (see e.g. \cite{Schwartz}):
\begin{equation}\label{lambda-quadratic}
 \lambda_\infty=\sqrt{A(1+B^2)}-1,\quad \Gamma_\infty(x,y)=\exp\left(-\frac 12\sqrt{\frac A{1+B^2}}(y-Bx)^2\right).
\end{equation}

 If the confined growth rate is the prototype example \eqref{r-quadratic}
with $\ep>0$, the principal eigenvalue $\lambda _\infty ^\ep$ can
also be explicitly computed, but the formula is more complicated.
One can however notice that we then have  $\lambda _\infty ^\ep
\to \lambda _\infty=\sqrt{A(1+B^2)}-1$, as $\ep \to 0$.
\end{rem}

\subsection{Preliminary control of the tails}

Our first result states that, in the confined and unconfined cases,
any global nonnegative solution of the Cauchy problem \eqref{eq_n}
has exponentially decaying tails. Notice that, in the mixed case,
Lemma \ref{lem:tails2} will provide an analogous estimate.

\begin{lem}[Exponential decay of tails]\label{lem:tails}
Assume that $r\in L^\infty_{loc}(\R^2)$ and $K\in
L^\infty((0,\infty)\times\R^3)$ satisfy \eqref{Assumption_r} and
\eqref{Assumption_K} respectively. Let Assumption
\ref{hyp-confined} or \ref{hyp-unconfined} hold. Assume that $n_0$
satisfies \eqref{initial-data}. Then, there exist $C>0$ and $\mu
>0$ such that, for any global nonnegative solution $n$ of
\eqref{eq_n},
\begin{equation}\label{tails} 0\leq n(t,x,y)\leq \begin{cases}
C e^{-\mu (|x-ct|+|y|)} &\text{ under Assumption \ref{hyp-confined},}\vspace{3pt}\\
Ce^{-\mu|y-B(x-ct)|} &\text{ under Assumption
\ref{hyp-unconfined}},
\end{cases}
\end{equation}
for all $t\geq 0$, $x\in \R$, $y \in \R$.
\end{lem}

\begin{proof}  Let us first work under Assumption \ref{hyp-unconfined}. If we define the mass $N(t,x):=\int _\R
n(t,x,y)\,dy$, an integration of \eqref{eq_n} along the variable
$y$ provides the inequality
$$\partial_t N-\partial_{xx} N\leq \left(r_{max} - k^- N\right)N.$$
Since $N(0,x)=\int _\R n_0(x,y)\,dy\leq \frac{2C_0}{\mu_0}$, it
then follows from the maximum principle that the mass is uniformly
bounded:
\begin {equation}\label{bornel1}
 N(t,x)\leq N_\infty:=\max\left(\frac{2C_0}{\mu_0},\frac{r_{max}}{k^-}\right).
\end {equation}
Since $N(t,x)=\int _\R n(t,x,y)\,dy$ is bounded in $\{(t,x):\,
t\geq 0, x\in \R\}$, it follows from Assumption
\ref{hyp-unconfined} and \eqref{Assumption_K} that there is $M>0$ such that
$$
\left|r(x-ct,y)-\int_{\R}K(t,x,y,y') n(t,x,y')\,dy'\right|\leq M,
$$
in $\Omega _{R+1}:=\{(t,x,y):\,(x-ct,y)\in
S_{R+1}\}=\{(t,x,y):\,|y-B(x-ct)|<R+1\}$. Hence $n(t,x,y)$ solves
a linear reaction diffusion equation with bounded coefficients on
$\Omega_{R+1}$. As a result, we can apply the parabolic Harnack
inequality (see \cite[page 391]{Eva} for instance). To do so let
us first choose $\ep:=(Bc)^{-1}$ so that, for any $T>0$,
$[T,T+\ep]\times {\{(x,y):\,|y-B(x-cT)|\leq R\}} \subset
\Omega_{R+1}$. Then, by the Harnack inequality, there is $C>0$
such that, for all $T>0$,
$$
\max_{(x,y), |y-B(x-cT)|\leq R}n(T,x,y)\leq C\min_{(x,y),
|y-B(x-cT)|\leq R}n(T+\ep,x,y).
$$
It then follows that
\begin{equation}\label{borneinf}
 \max_{(x,y), |y-B(x-cT)|\leq R}n(T,x,y)\leq \frac C{2R} \int _\R n(T+\ep,x,y)\,dy=\frac C{2R} N(T+\ep,x)\leq \frac C{2R}
 N_\infty.
\end{equation}
Hence, the population $n(t,x,y)$ is bounded by $\frac C{2R}
 N_\infty$ in $\Omega
_R:=\{(t,x,y):\,|y-B(x-ct)|\leq R\}$.

Next, to handle the remaining region ${\Omega
_R}^c=\{(t,x,y):\,|y-B(x-ct)|>R\}$, let us define
$$
\varphi(t,x,y):=\kappa e^{-\mu(|y-B(x-ct)|-R)},
$$
which, in ${\Omega _R}^c$, satisfies
\begin{eqnarray*}
\partial _t \varphi-\partial_{xx} \varphi -\partial_{yy}\varphi
-r(x-ct,y)\varphi&=&(\pm \mu Bc- \mu ^2 B^2 - \mu ^2-r(x-ct,y))
\varphi \\
&\geq& (\pm \mu Bc- \mu ^2 B^2 -\mu ^2+\delta)\varphi
\end{eqnarray*}
by Assumption \ref{hyp-unconfined}. Choosing $\mu >0$ small enough
makes $\varphi$ a super-solution on ${\Omega _R}^c$, whereas $n$
is a sub-solution. If we choose $\kappa=\max(C_0,\frac C{2R}
N_\infty)$ and $\mu\in (0,\mu _0)$, then we enforce $n(t,x,y)\leq
\varphi(t,x,y)$ on  $\{0\}\times\mathbb R^2$ (see
\eqref{initial-data}) and on the parabolic lateral boundary of
${\Omega_R}^c$ (see \eqref{borneinf}). It follows from the
parabolic maximum principle that $n(t,x,y)\leq \varphi(t,x,y)$ in
${\Omega _R}^c$, which concludes the proof of the lemma under
Assumption \ref{hyp-unconfined}.

 When Assumption \ref{hyp-confined}
holds, arguments are very similar. It suffices to select a $\delta
>0$ and an associated $R>0$ such that Assumption \ref{hyp-confined} holds, then to take $\Omega _{R+1}:=\{(t,x,y):\,(x-ct,y)\in
T_{R+1}\}=\{(t,x,y):\,|x-ct|+|y|< R+1\}$ and, finally, to use
$\varphi(t,x,y):=\kappa e^{-\mu(|x-ct|+|y|-2R)}$ on the remaining
region. Details are omitted.
\end{proof}

\subsection{Preliminary Harnack-type estimate}
\label{subsec:harnack}

We present here our refinement of the parabolic Harnack inequality. We already discussed in subsection \ref{ss:estimating} the relevance of Theorem \ref{th:harnack-mod}, which is also of independent interest.

Let $\Omega\subset\mathbb R^N$ an open set with $N\geq 1$. We consider a solution $u(t,x)$, $t>0$, $x\in \Omega$, of a linear parabolic equation, namely
\begin{equation}\label{eq:harnack}
\partial_t u(t,x)-\sum_{i,j=1}^N a_{i,j}(t,x)\partial_{x_i,x_j} u(t,x)-\sum_{i=1}^Nb_i(t,x)\partial_{x_i}u(t,x)= f(t,x)u(t,x),
\end{equation}
where the coefficients are bounded, and $(a_{i,j})_{i,j=1,\dots,N}$ is uniformly elliptic. Let us first recall the parabolic Harnack inequality, as proved by Moser \cite{Moser}.

\begin{theo}[Parabolic Harnack inequality, \cite{Moser}]\label{thm:Harnack}
Let us assume that all the coefficients $(a_{i,j})_{i,j=1,\dots,N}$, $(b_i)_{i=1,\dots,N}$, $f$ belong to  $L^\infty_{loc}\left((0,\infty)\times \Omega\right)$, where $\Omega$ is an open set of $\mathbb R^N$, and that $(a_{i,j})_{i,j=1,\dots,N}$ is uniformly positive definite on $\Omega$. Let $\tau>0$ and $0<R<R'$.

There exists $C_H>0$ such that for any $(\bar t,\bar x)\in (2\tau,\infty)\times \mathbb R^N$ such that $B_{\mathbb R^N}(\bar x, R')\subset\Omega$, and for any nonnegative (weak) solution 
$u\in H^1 \left((0,\infty)\times \Omega\right)$ of \eqref{eq:harnack} on $(\bar t-2\tau,\bar t)\times \Omega$, 
\begin{equation}\label{eq:theoHarnack}
\max_{x\in B(\bar x,R)}u(\bar t-\tau, x)\leq C_H \min_{x\in B(\bar x,R)}u(\bar t, x).
\end{equation}
\end{theo}

\begin{rem}\label{Rk:Harnack} Let us emphasize that the time shift $\tau >0$ is necessary for \eqref{eq:theoHarnack} to hold. To see this, consider for $N=1$ the family $u(t,x)=\frac 1{\sqrt{4\pi t}}e^{-\frac{(x+ x_0)^2}{4t}}$ ($x_0\in \R$)  of solutions to the Heat equation. Then, for $\bar t=1$ and $\bar x=0$, we have 
\[\frac{\max_{x\in B(0,R)}u(1, x)}{\min_{x\in B(0,R)}u(1, x)}\geq \frac{u(1,0)}{u(1,R)}=e^{\frac{(R+ x_0)^2-{ x_0}^2}{4}}=e^{\frac{2R x_0+R^{2}}{4}},\]
which is not bounded from above as $x_0\to \infty$. Hence, estimate \eqref{eq:theoHarnack} cannot hold with $\tau=0$. 
\end{rem}

Nevertheless, provided the solution $u$ is uniformly bounded, we can derive from Theorem~\ref{thm:Harnack} the following refinement, where no time shift is required.

\begin{theo}[A refinement of Harnack inequality]\label{th:harnack-mod}
Let us assume that all the coefficients $(a_{i,j})_{i,j=1,\dots,N}$, $(b_i)_{i=1,\dots,N}$, $f$ belong to  $L^\infty_{loc}\left((0,\infty)\times \mathbb R^N\right)$, and that $(a_{i,j})_{i,j=1,\dots,N}$ is uniformly positive definite on $\mathbb R^N$. Let also $\omega\subset \mathbb R^N$, and assume that for any $R'>0$, there exists  $K>0$ such that, for all $1\leq i,j\leq N$,
\begin{equation}
\label{def-K}
a_{i,j}(t,x)\leq K, \; b_i(t,x)\leq K, \; f(t,x)\leq K\; \text{ a.e. on } (0,\infty)\times \omega+B_{\mathbb R^N}(0,R').
\end{equation}
 Let  $R>0$, $\delta >0$, $U>0$,   $\bar t>0$ and $\rho>0$.

There exists $C>0$ such that for  any $\bar x\in\mathbb R^N$ satisfying $\textrm{d}_{\mathbb R^N}(\bar x,\omega)\leq\rho$, and any nonnegative (weak) solution 
$u\in H^1 \left((0,\infty)\times \mathbb R^N\right)$ of \eqref{eq:harnack} on $(0,\infty)\times \mathbb R^N$, and such that $\|u\|_{L^\infty(\mathbb R^N)}\leq U$, we have
\begin{equation}\label{eq:theoHarnack2}
\max_{x\in B(\bar x,R)}u(\bar t, x)\leq C\min_{x\in B(\bar x,R)}u(\bar t, x)+\delta.
\end{equation}
\end{theo}

\begin{proof} Let us assume without loss of generality that $\bar t=2$ and $\bar x=0$. We introduce
\begin{equation*}
\phi(t,x):=e^{K (t-1)}\left[\max_{|x|\leq \alpha
R}u\left(1,x\right)+\frac{2\|u\|_\infty}{(\alpha
R)^2}
NK(1+\alpha R)
(t-1)+\frac{\|u\|_\infty}{(\alpha
R)^2}|x|^2\right],
\end{equation*}
where $\alpha>1$ is to be determined later. We aim at applying the parabolic comparison principle on the domain $(1,2)\times B(0,\alpha R)$. We have
$\phi\left(1,x\right)\geq
u\left(1,x\right)$ for $|x|\leq \alpha R$, and $\phi(t,x)\geq \Vert u \Vert _{L^\infty(\Omega)}\geq u(t,x)$ for $1\leq t \leq 2$, $|x|=\alpha R$. A simple computation shows
$$
\partial_t\phi -\sum_{i,j=1}^Na_{i,j}\partial_{x_i,x_j}\phi-\sum_{i=1}^Nb_i\partial_{x_i} \phi-K\phi=\frac{2\Vert u \Vert _\infty}{(\alpha R)^{2}}e^{K(t-1)}\left[NK(1+\alpha R)-\sum _{i=1}^N a_{i,i}(t,x)-\sum _{i=1}^N b_i (t,x)x_i\right]
$$
which, in view of \eqref{def-K}, is nonnegative
on $(1,2)\times B(0,\alpha R)$. Since $\partial _t u -\sum_{i,j=1}^Na_{i,j}\partial_{x_i,x_j} u
-\sum_{i=1}^Nb_i\partial_{x_i}u-K u\leq 0$, the comparison principle yields
$u(2,\cdot)\leq \phi(2,\cdot)$ on $B(0,\alpha R)$. In particular we
have
\begin{eqnarray}
\max_{|x|\leq R}u(2,x)&\leq &\max_{|x|\leq R}\phi(2,x)\nonumber\\
&\leq& e^{K }\left(\max_{|x|\leq \alpha
R}u\left(1,x\right)+\frac{2\|u\|_\infty}{(\alpha R)^2}NK(1+\alpha R)
+\frac{\|u\|_\infty}{\alpha^2}\right)\nonumber\\
&\leq&e^{K}\max_{|x|\leq \alpha R}u\left(1,x\right)+\frac{e^K U}{\alpha ^2}\left(\frac{2NK(1+\alpha R)}{R^2}+1\right)\nonumber\\
&\leq&e^{K}\max_{|x|\leq \alpha R}u\left(1,x\right)+\delta,
\label{bidule1harnack}
\end{eqnarray}
provided we select $\alpha>1$ large enough. Thanks to \eqref{def-K}, we can then apply Theorem~\ref{thm:Harnack} with $\tau:=1$ and $\Omega:=\omega+B_{\mathbb R^N}(0,2\alpha R)$, to get that there exists a constant $C_H=C_H(R,\alpha)>0$ such that
\begin{equation}
\max_{|x|\leq \alpha R}u\left(1,x\right)\leq C_H\min_{|x|\leq \alpha R}u\left(2,x\right)\leq C_H\min_{|x|\leq  R}u\left(2,x\right).\label{bidule2harnack}
\end{equation}
Theorem \ref{th:harnack-mod} follows from \eqref{bidule1harnack} and \eqref{bidule2harnack}.
\end{proof}

\section{The confined case}\label{s:survival}

In this section, we consider the confined case, namely Assumption
\ref{hyp-confined}, for which the growth rate $r(x,y)$ is positive for a bounded set of points $(x,y)$ only. We discuss the extinction or the survival of the
population, defining a critical speed $c^*$ for the climate shift
by
\begin{equation}\label{critical-speed}
c^*:=\begin{cases}2\sqrt{-\lambda_\infty} & \text{ if }
\lambda_\infty< 0\vspace{3pt}\\
-\infty & \text{ if } \lambda_\infty \geq0,
\end{cases}
\qquad \text{the critical speed in the confined case,}
\end{equation}
where $\lambda _\infty$ is the principal eigenvalue defined by
\eqref{vp-pb}. In the whole section, we are then equipped with
$\lambda _\infty$ and $\Gamma _\infty(x,y)$ satisfying
\eqref{vp-pb}.

We introduce  two changes of variable that will be very convenient
in the sequel. Precisely, we define
\begin{equation}\label{change-var}
\tilde n(t,x,y):=n(t,x+ct,y), \; u(t,x,y):=e^{\frac {cx}2}\tilde
n(t,x,y)
\end{equation}
so that the equation \eqref{eq_n} is recast as
\begin{align}
& \partial_t \tilde n(t,x,y)-c\partial_x \tilde n(t,x,y)-\partial_{xx}\tilde n(t,x,y)-\partial_{yy}\tilde n(t,x,y)\nonumber\\
&\quad=\left(r(x,y)-\int_{\mathbb R}K(t,x+ct,y,y')\tilde
n(t,x,y')\,dy'\right)\tilde n(t,x,y),\label{eq-tilde-n}
\end{align}
or as
\begin{align}
& \partial_t u(t,x,y)-\partial_{xx} u(t,x,y)-\partial_{yy} u(t,x,y)\nonumber\\
&\quad = \left(r(x,y)-\frac{c^2}4-\int_{\mathbb
R}K(t,x+ct,y,y')e^{-\frac{cx}2}u(t,x,y')\,dy'\right)u(t,x,y).\label{eq_u}
\end{align}

\subsection{Extinction}

In this subsection, we show extinction of the population for rapid
climate shifts $c>c^*$. Since extinction comes from the linear
part of the equation \eqref{eq_n}, the nonlocal term will not be a
problem here. If follows that if $c>c^*$, the extinction for any
reasonable initial population (see Proposition
\ref{prop:extinction2}) can be proven thanks to an argument
similar to the one in \cite[Proposition 1.4]{Ber-Ros2}.
Nevertheless, this argument does not provide any information on
the speed of extinction. If  we further assume sufficient decay of
the tails of the initial data (see Proposition
\ref{prop:extinction}), we can show that the extinction is
exponentially fast. We start with this last situation.

\begin{prop}[Extinction with initial control of the tails]\label{prop:extinction}
Assume that $r\in L^\infty_{loc}(\R^2)$ and $K\in
L^\infty((0,\infty)\times\R^3)$ satisfy \eqref{Assumption_r} and
\eqref{Assumption_K} respectively. Let Assumption
\ref{hyp-confined} hold. Assume that $n_0$ satisfies
\eqref{initial-data}. If $c> c^*$ and the initial population
satisfies
\begin{equation}\label{queues}
M:=\sup _{(x,y)\in \R^2}
\frac{e^{\frac{cx}2}n_0(x,y)}{\Gamma_\infty(x,y)} <\infty,
\end{equation}
then any global nonnegative solution $n(t,x,y)$ of \eqref{eq_n}
satisfies
\begin{equation}\label{controle1}\sup _{(x,y)\in \R^2}
\frac{e^{\frac{cx}2}n(t,x+ct,y)}{\Gamma_\infty(x,y)}= \mathcal
O\left(e^{\left(-\lambda_\infty-\frac {c^2}4\right)t}\right)\to
0,\, \text{ as } t\to \infty, \end{equation} and, for some $\gamma
_0
>0$,
\begin{equation}\label{controle2}
\sup_{x\in\mathbb R} \,\int _\R n(t,x,y)\,dy= \mathcal
O(e^{-\gamma _0 t}) \to 0,\, \text{ as } t\to\infty.
\end{equation}

\end{prop}

\begin{proof}
We consider
$$\phi(t,x,y):=M e^{\left(-\lambda_\infty-\frac {c^2}4\right)t}\Gamma_\infty(x,y),$$
which satisfies
$$
\partial_t \phi(t,x,y)-\partial_{xx}\phi(t,x,y)-\partial_{yy} \phi(t,x,y)
=\left(r(x,y)-\frac{c^2}4\right)\phi(t,x,y).
$$
In view of \eqref{eq_u}, $u(t,x,y)=e^{\frac{cx}2}n(t,x+ct,y)$ is a
sub-solution of the above equation. Since the definition of $M$
implies $u(0,x,y)\leq \phi(0,x,y)$, it  follows from the parabolic
maximum principle that $u(t,x,y)\leq \phi(t,x,y)$ for any time
$t\geq0$, which proves \eqref{controle1}.

To prove \eqref{controle2}, we will need
\begin{equation}\label{postponed}
\rho:=\sup _{x\in \R} \, \int _{\R} \Gamma _\infty(x,y)\,dy
<\infty,
\end{equation}
whose proof is postponed. If $c=0$, \eqref{controle2} follows from
\eqref{controle1} and \eqref{postponed}. If $c>0$, combining
\eqref{controle1} and the control of the tails \eqref{tails} we
obtain, for $\alpha
>0$ to be selected,
\begin{eqnarray*}
\sup_{x\in\mathbb R}\,\int_\R n(t,x,y)\,dy&=&
 \sup_{x\in\mathbb R}\,\int_\R n(t,x+ct,y)\,dy\\
 &\leq& \sup_{x\leq -\alpha t}\,\int_\R n(t,x+ct,y)\,dy+\sup_{x\geq -\alpha t}\int_\R n(t,x+ct,y)\,dy\\
&\leq& \sup_{x\leq -\alpha t}\,\int_\R C
e^{-\mu(|x|+|y|)}\,dy+\sup_{x\geq -\alpha t}\int_\R
e^{-\frac{cx}2}Me^{\left(-\lambda_\infty-\frac {c^2}4\right)t}
\Gamma _\infty (x,y)\,dy\\
&\leq&\frac{2C}\mu e^{-\mu\alpha t}+M\rho
e^{\left(-\lambda_\infty-\frac {c^2}4+\frac{c\alpha}2\right)t},
\end{eqnarray*}
which proves \eqref{controle2} by selecting $\alpha >0$ small
enough so that $-\lambda_\infty-\frac {c^2}4+\frac{c\alpha}2<0$.

To conclude, let us now prove \eqref{postponed}. Select $\delta
>0$ such that $\delta \geq 2 \lambda _\infty$. For this $\delta
>0$, select $R>0$ as in Assumption \ref{hyp-confined} so that --- in
view of equation \eqref{vp-pb}--- the principal eigenfunction
satisfies $-\partial _{xx}\Gamma _\infty-
\partial _{yy} \Gamma _\infty \leq -\frac \delta 2 \Gamma
_\infty$ in $\{(x,y):\,|x|\geq R,\,|y|\geq R \}$ and $\Vert \Gamma
_\infty \Vert _\infty \leq 1$. It therefore follows from the
elliptic comparison principle that
\begin{equation}\label{etoile}
 \Gamma _\infty (x,y)\leq e^{-\sqrt{\frac \delta
2}\left(|y|-R\right)}e^{-\sqrt{\frac \delta
2}\left(|x|-R\right)}\; \text{ in } \{(x,y):\,|x|\geq R\,|y|\geq R
\}.
\end{equation}
Therefore we have, for all $(x,y)\in \R ^2$,
\begin{equation}\label{bornegammainf}
\Gamma _\infty (x,y) \leq C e^{-\sqrt{\frac \delta 2}(|x|+|y|)},
\end{equation}
which implies \eqref{postponed}.\end{proof}

\begin{rem} Observe that, in the unconfined case, the estimate \eqref{controle1} remains valid
but the control of the tails \eqref{tails} under Assumption
\ref{hyp-unconfined} does not imply \eqref{controle2}. Roughly
speaking, in the unconfined case, even if $c>c^*$ there is a
possibility that the population survives, but migrates towards
large $x$ at a speed $\omega\in (0,c)$ different from the climate
change speed $c$. Therefore we need a different criterion that
will be discussed in Section \ref{s:propagation}.
\end{rem}

\begin{prop}[Extinction for general initial data]\label{prop:extinction2}
Assume that $r\in L^\infty_{loc}(\R^2)$ and $K\in
L^\infty((0,\infty)\times\R^3)$ satisfy \eqref{Assumption_r} and
\eqref{Assumption_K} respectively. Let Assumption
\ref{hyp-confined} hold. Assume that $n_0\in L^\infty(\R ^2)$ and
that $\sup _{x\in \R}\int _\R n_0(x,y)\,dy<\infty$. If $c> c^*$
then any global nonnegative solution $n(t,x,y)$ of \eqref{eq_n}
satisfies
\begin{equation}\label{extinction-henri}
\lim _{t\to \infty} n(t,x,y)=0,
\end{equation}
uniformly with respect to $(x,y)\in \R^2$.
\end{prop}

\begin{proof} It is equivalent and more convenient to prove \eqref{extinction-henri} for
$\tilde n(t,x,y)=n(t,x+ct,y)$ which solves \eqref{eq-tilde-n}.
Since $r(x,y)\to -\infty$ as $|x|+|y|\to \infty$, we first define, for some $R>0$, the cut-off function
\begin{equation*}\label{cut-off} r_{cut}(x,y):= \begin{cases}
r(x,y) & \text{ if } (x,y)\in B_{R}\\
\sup _{(x,y)\notin B_R} r(x,y) & \text{ if } (x,y) \notin B_{R},
\end{cases}
\end{equation*}
which is larger than $r(x,y)$ and has the advantage of being bounded. Let
$\left(\lambda_{cut},\Gamma_{cut}\right) \in \R \times
C^\infty(\R)$ solve the
generalized principal eigenvalue problem
\begin{equation}\label{vp-pb-cut}
\begin{cases}
-\partial_{xx}
\Gamma_{cut}(x,y)-\partial_{yy}\Gamma_{cut}(x,y)-r_{cut}(x,y)\Gamma_{cut}(x,y)=
\lambda_{cut}\Gamma_{cut}(x,y)  \quad\text{ for all  } (x,y)\in \R^2\\
\Gamma_{cut}(x,y) >0\quad \text{ for all } (x,y)\in \R^2,\quad
\Vert \Gamma _{cut} \Vert _\infty=1.
 \end{cases}
 \end{equation}
We claim that
\begin{equation}\label{claim-ne}
\lambda _{cut} \nearrow \lambda
_\infty
, \text{ as } R\to \infty.
\end{equation}
Since arguments are rather  classical (see \cite[Proposition
4.2]{ber-ham-ros} for instance), we only sketch the proof. Since
$r_{cut} \geq r$ we have $\lambda _{cut}\leq \lambda _{\infty}$.
Also, $\lambda _{cut}$ is increasing with respect to $R$. Hence
$\lambda  _{cut} \nearrow \tilde \lambda\leq \lambda_{\infty}$, as
$R\to \infty$. Assume, by way of contradiction, that $\tilde
\lambda <\lambda _{\infty}$. Since $r_{cut}\to r$ locally
uniformly, it follows from the Harnack inequality, elliptic
interior estimates and a diagonal extraction that we can construct
a function $\gamma >0$ such that $-\partial _{xx} \gamma -\partial
_{yy} \gamma =(r+\tilde \lambda)\gamma $, and $\gamma$ is then a
subsolution of the equation satisfied by $\Gamma _{\infty}$ (see
the first line of \eqref{vp-pb}). Outside of a large ball, the
zero order term of the equation, namely  $r+\lambda _\infty$, is
negative so the elliptic comparison principle applies outside of a
large ball. This enables us to define
$$
\ep_0:=\sup\{\ep>0:\, \ep\gamma (x,y)\leq \Gamma _\infty (x,y), \forall (x,y) \in \R ^2\}.
$$
Hence $\psi:=\Gamma _\infty -\ep _0\gamma$ has a zero minimum at some point $(x_0,y_0)$ and therefore
$0\leq (\partial _{xx}\psi+\partial _{yy}\psi+r\psi)(x_0,y_0)\leq (\tilde \lambda -\lambda _\infty)\Gamma _\infty(x_0,y_0)<0$.
This contradiction proves the claim \eqref{claim-ne}. As a result, we can choose $R>0$ large enough so that $r\leq -1$ on the
complement of $B_{R/2}$, and such that $-\frac {c^2} 4<\lambda
_{cut}$ (we recall that $c>c^*$ implies $-\frac {c^2}4<\lambda_\infty$).

By the parabolic comparison principle,  we have
\begin{equation}\label{controle-par-w}
0\leq \tilde n(t,x,y) \leq w(t,x,y),
\end{equation}
where $w(t,x,y)$ is the solution of the following linear problem --- obtained by dropping 
the nonlocal term and replacing $r(x,y)$ 
by $r_{cut}(x,y)$ in \eqref{eq-tilde-n}---
\begin{equation}\label{eq-w}
 \left\{\begin{array}{lll}
          \partial_t w-c\partial_x  w-\partial_{xx}  w-\partial_{yy} w=
          r_{cut}(x,y)w\vspace{3pt}\\
      w(0,x,y)=w_0(x,y):=M+Ae^{-\frac c 2 x}\Gamma _{cut}(x,y),
              \end{array}
\right.
\end{equation}
where $M:=\Vert n_0 \Vert _{L^\infty}$. Recalling that  $\frac{c^2}{4}+\lambda_{cut}>0$ and that $r_{cut}\leq 0$ on the complement of
$B_{R/2}$, we can choose $A>0$ large enough so that
$$
(\partial _t-c\partial _x -\partial _{xx}-\partial
_{yy}-r_{cut}(x,y))w_0(x,y)=A\left(\frac{c^2}4+\lambda
_{cut}\right)e^{-\frac c 2 x}\Gamma _{cut} (x,y)-r_{cut}(x,y)M
>0,$$
for all $(x,y)\in\mathbb R^2$. In other words, the initial data in
\eqref{eq-w} is a super solution of the parabolic equation in
\eqref{eq-w}, which implies that $t\mapsto w(t,x,y)$ is
decreasing for any $(x,y)\in\mathbb R^2$. As a result there is a function $\bar w(x,y)$ such
that
$$
w(t,x,y)\to \bar w(x,y)\quad\text{ as } t\to \infty, \qquad 0\leq
\bar w(x,y)\leq M+Ae^{-\frac c 2 x}\Gamma _{cut}(x,y).
$$

Let us prove that the above pointwise convergence actually holds
locally uniformly w.r.t. $(x,y)$, and that $\bar w$ solves
\begin{equation}\label{eq-pour-w-barre}
-c\partial _x \bar w -\partial _{xx}\bar w -\partial _{yy}\bar w
-r_{cut}(x,y)\bar w=0.
\end{equation}
Let $R'>0$ be given. Let $(t_n)$  be an arbitrary sequence such
that $t_n\to \infty$, and define $w_n(t,x,y):=w(t+t_n,x,y)$. $w_n$ then solves
$$
\partial _t w_n-c\partial _x w_n-\partial _{xx}w_n-\partial
_{yy}w_n-r_{cut}(x,y)w_n=0,
$$
that is a linear equation whose coefficients are bounded on
$(0,\infty)\times \R ^2$ uniformly w.r.t. $n$. By the interior
parabolic estimates \cite[Section VII]{Lie}, for a fixed $p>2$,
there is a constant $C_{R'}>0$ such that
$$
\Vert w_n \Vert _{W^{1,2}_{p}((1,2)\times B_{R'})}\leq C_{R'} \Vert
w_n\Vert_{L^p((0,3)\times B_{2{R'}})}\leq C_{R'} (3|B_{2{R'}}|)^{1/p}\Vert
w\Vert _{L^\infty((0,\infty)\times \R ^2)}<\infty.
$$
Since $p>2$, there is $0<\alpha<1$ such that the injection
$W^{1,2}_{p}((1,2)\times B_{R'}) \hookrightarrow
C^{\frac{1+\alpha}2,1+\alpha}\left([1,2]\times
\overline{B_{R'}}\right)$ is compact. Therefore there is a
subsequence $w_{\varphi(n)}$ which converges in
$C^{\frac{1+\alpha}2,1+\alpha}\left([1,2]\times
\overline{B_{R'}}\right)$, and converges weakly in
$W^{1,2}_{p}((1,2)\times B_{R'})$. The limit of $w_{\varphi(n)}$ has
to be $\bar w$, which is independent on the sequence
$t_n\to\infty$ and the extraction $\varphi$. Therefore
$w(t,\cdot,\cdot)\to \bar w(\cdot,\cdot)$,  in both
$C^{\frac{1+\alpha}2,1+\alpha}$ and $W^{1,2}_{p}$ locally
uniformly. Hence, the convergence $w(t,x,y)\to\bar w(x,y)$, as
$t\to \infty$, is locally uniform w.r.t. $(x,y)$. On the other
hand, the weak convergence in $W^{1,2}_{p}$ allows to pass to the
limit in equation \eqref{eq-w}, so that $\bar w$ actually solves
\eqref{eq-pour-w-barre}.

We claim that $\bar w\equiv 0$. Indeed define $\psi(t,x):=e^{\frac
c 2 x}\bar w(x,y)$ which solves
\begin{equation}\label{eq-pour-psi}
-\partial _{xx}\psi -\partial _{yy}\psi
-r_{cut}(x,y)\psi=-\frac{c^2}4\psi.
\end{equation}
Multiplying equation \eqref{vp-pb-cut} by $\psi$, equation
\eqref{eq-pour-psi} by $\Gamma _{cut}$ and integrating the
difference over the ball $B_{R'}$, we get
$$
\int _{B_{R'}} (\Gamma _{cut} \Delta \psi -\psi \Delta \Gamma _{cut}
)=\left(\frac{c^2}4+\lambda _{cut}\right)\int _{B_{R'}}\psi \Gamma
_{cut}.
$$
Applying the Stokes theorem implies
\begin{equation}\label{stokes}
\int _{\partial B_{R'}} \left(\Gamma _{cut} \frac{\partial
\psi}{\partial \nu} -\psi \frac{\partial \Gamma _{cut}}{\partial
\nu}\right)=\left(\frac{c^2}4+\lambda _{cut}\right)\int _{B_{R'}}\psi
\Gamma _{cut}.
\end{equation}
We want to let ${R'}\to \infty$ in the left hand side member. It is
easily seen that $\Gamma _{cut}$ also satisfy estimate
\eqref{bornegammainf} so that  $\Vert \Gamma _{cut}\Vert
_{L^\infty(\partial B_{R'})}\to 0$, as ${R'}\to \infty$. Since
$$
-\Delta\Gamma _{cut}=f(x,y):=(r_{cut}(x,y)+\lambda
_{cut})\Gamma _{cut}(x,y),
$$
the interior elliptic estimates \cite[Theorem 9.11]{Gil-Tru}
provide, for a fixed $p>2$,  some $C>0$ such that, for all
$P_0=(x_0,y_0)\in\R ^2$,
\begin{eqnarray*}
\Vert \Gamma_{cut}\Vert_{W^2_p(B(P_0,1))}&\leq& C\left(\Vert
\Gamma _{cut}\Vert _{L^p(B(P_0,2))}+\Vert f \Vert
_{L^p(B(P_0,2))}\right)\\
&\leq & C|B_2|^{1/p}\left(1+(\Vert r_{cut}\Vert_{L^\infty(\R
^2)}+|\lambda _{cut}|)\right)\Vert \Gamma _{cut}\Vert
_{L^\infty(\R ^2)}<\infty,
\end{eqnarray*}
where we have used the analogous of \eqref{bornegammainf} for
$\Gamma _{cut}$. Since $p>2$, there is $0<\alpha<1$ such that the
injection $W^{2}_{p}(B(P_0,1)) \hookrightarrow
C^{1+\alpha}(\overline{B(P_0,1)})$ is compact and therefore $\Vert
\nabla \Gamma _{cut}\Vert _{L^\infty(\overline{B(P_0,1)})}\leq C$,
for some $C$ independent of $P_0 \in \R^2$. As a result $\nabla
\Gamma _{cut} \in L^\infty(\R^2)$. Let us now deal with $\psi$ and
$\nabla \psi$ by using rather similar arguments. First, in view of
\eqref{eq-pour-psi}, we have $-\Delta\psi\leq- \psi$ outside
$B_{R'}$. It therefore follows from the elliptic comparison principle
that $\psi(x,y) \leq M e^{-(|x|+|y|)}$ outside $B_{R'}$, where
$M:=\Vert \psi \Vert _{L^\infty(\partial B_{R'})}\left(\min _{(x,y)\in
\partial B_{R'}}e^{-(|x|+|y|)}\right)^{-1}$. As a result
$\psi$ also satisfies an estimate analogous to
\eqref{bornegammainf}, precisely $\psi(t,x)\leq C e^{-(|x|+|y|)}$.
In particular $\Vert \psi \Vert _{L^\infty(\partial B_{R'})} \to 0$,
and we can reproduce the above argument to deduce that $\nabla
\psi \in L ^\infty(\R^2)$. Hence, from the estimates on $\Gamma_
{cut}$, $\nabla \Gamma _{cut}$, $\psi$, $\nabla \psi$, the left
hand side of \eqref{stokes} tends to zero as ${R'}\to \infty$.
Since $\frac{c^2}4+\lambda _{cut}>0$, this implies $\psi\equiv \bar w\equiv 0$.

As a result we have $w(t,x,y)\to 0$ as $t\to \infty$, locally
uniformly w.r.t. $(x,y)\in \R^2$. We claim that this convergence
is actually uniform w.r.t. $(x,y)\in \R^2$. Indeed, assume by
contradiction that there is $\ep >0$, $t_n \to \infty$,
$|x_n|+|y_n|\to \infty$, such that $w(t_n,x_n,y_n)\geq \ep$.
Define $w_n(t,x,y):=w(t,x+x_n,y+y_n)$ which solves
\begin{equation}\label{eq-shift-espace}
\partial _t w_n-c\partial _x w_n-\partial _{xx}w_n-\partial
_{yy}w_n=r_{cut}(x+x_n,y+y_n)w_n,
\end{equation}
that is a linear equation whose coefficients are bounded on
$(0,\infty)\times \R ^2$ uniformly w.r.t. $n$, since $r_{cut}\in
L^\infty(\R^2)$. Using the interior parabolic estimates and
arguing as above, we see that (modulo extraction) $w_n(t,x,y)$
converge to some $\theta(t,x,y)$ strongly in
$C^{\frac{1+\alpha}2,1+\alpha}_{loc}((0,\infty)\times \R ^2)$, and
weakly in $W^{1,2}_{p, loc}((0,\infty)\times \R ^2)$. Hence,
letting $n\to \infty$ into \eqref{eq-shift-espace}, we have
\begin{equation*}
\partial _t \theta -c\partial _x \theta -\partial _{xx} \theta -\partial
_{yy}\theta \leq -\theta,
\end{equation*}
so that $\theta (t,x,y)\leq C e^{-t}$ by the comparison principle.
In particular
$$
0=\lim _{t\to \infty}\theta (t,0,0)=\lim _{t\to \infty}\lim _n
w(t,x_n,y_n)\geq \liminf _n w(t_n,w_n,y_n)\geq \ep, $$ that is a
contradiction. Hence, $w(t,x,y)\to 0$ as $t\to \infty$, uniformly
w.r.t. $(x,y)\in \R^2$. In view of \eqref{controle-par-w}, the
same holds true for $\tilde n(t,x,y)$.
\end{proof}

\subsection{Persistence}

For slow climate shifts $0\leq c< c^*$, our result of persistence of
the population reads as follows.

\begin{theo}[Survival]\label{th:survival}
Assume that $r\in L^\infty_{loc}(\R^2)$ and $K\in
L^\infty((0,\infty)\times\R^3)$ satisfy \eqref{Assumption_r} and
\eqref{Assumption_K} respectively. Let Assumption
\ref{hyp-confined} hold. Assume that $n_0\not\equiv 0$ satisfies
\eqref{initial-data}. If $0\leq c < c^*$,  then, for any
nonnegative solution $n$ of \eqref{eq_n}, there exists a
nonnegative function $h$ such that
\begin{equation}\label{reste-qqch}
\int _{\R ^2} h (x,y)\,dx\,dy>0,\quad \int _\R h(0,y)\,dy>0,
\end{equation}
and
$$
\tilde n (t,x,y)=n(t,x+ct,y)\geq h(x,y)\,\text{ for all } t\geq
1,\, x\in \R,\, y \in \R.
$$
\end{theo}

\begin{rem}
Before proving the theorem, observe that the above result cannot
hold for a speed $\tilde c$ different from $c$. Indeed it follows
from the control of the tails \eqref{tails} that
$$
n(t,\tilde c \,t,y)\leq Ce^{-\mu|c-\tilde c|t}e^{-\mu|y|},
$$
 and then, $\int
_\R Ce^{-\mu|c-\tilde c|t}e^{-\mu|y|}\,dy \to 0$, as $t\to
\infty$. This indicates that, as stated in the introduction, the
species needs to follow the climate shift to survive.
\end{rem}

\begin{proof}  For $R>0$ define the rectangle
$$
\Omega _R:=\left\{(x,y):\,|x|<\frac R {B+1}, |y|< R\right\},
$$
and denote by $(\lambda_R,\Gamma_R)\in\mathbb R\times
C^\infty\left(\overline{\Omega _R}\right)$ the solution of the
 principal eigenvalue problem
\begin{equation}\label{vp-pb-rectangle}
\begin{cases}
-\partial_{xx}
\Gamma_R(x,y)-\partial_{yy}\Gamma_R(x,y)-r(x,y)\Gamma_R(x,y)=
\lambda_R\Gamma_R(x,y)  \quad\text{ for all  } (x,y)\in \Omega _R\vspace{3pt}\\
\Gamma _R(x,y)=0 \quad\text{ for all } (x,y) \in \partial \Omega _
R\vspace{3pt}\\
\Gamma_R(x,y) >0\quad \text{ for all } (x,y)\in \Omega _R,\quad
\Vert \Gamma _R \Vert _\infty=1.
 \end{cases}
 \end{equation}
Since $\lambda_R$ converges to $\lambda_\infty$ as $R\to\infty$,
we can select $R>0$ large enough so that (observe that $0\leq
c<c^*$ reads as $\lambda _\infty<-\frac{c^2}4$)
\begin{equation}\label{inegalite}
\ep:=-\frac{c^2}4-\lambda _R >0.
\end{equation}

Next, the control of the tails \eqref{tails} in Lemma
\ref{lem:tails} implies, that for any $t\geq 0$ and any $x$ such
that $|x|\leq \frac M{B+1}$,
\begin{eqnarray}
\int_{|y'|\geq M}K(t,x+ct,y,y')\tilde n (t,x,y')\,dy'
&=& \int_{|y'|\geq M}K(t,x+ct,y,y')n(t,x+ct,y')\,dy'\nonumber\\
&\leq & Ck^+ \int _{|y'|\geq M} e^{-\mu(|x|+|y'|)}\,dy'\nonumber\\
&\leq & C k^+ \int _{|y'|\geq M}e^{-\mu
|y'|}\,dy'\nonumber\\
&=& \frac{2C k^+}\mu e^{-\mu M}< \frac \ep 2,\label{controle-loin}
\end{eqnarray}
provide we select $M>R$ large enough.

Recall that $\tilde n$ solves \eqref{eq-tilde-n}. Since $r\in
L^\infty(\Omega _{M+1})$ and since \eqref{bornel1} shows that the
nonlocal term is uniformly bounded, we can apply the parabolic
Harnack inequality: there exists $C_H>0$ such that, for all $t\geq
1$,
\begin{equation}\label{harnack}
\max_{(x,y)\in\overline{\Omega _M} }\tilde n\left(t-\frac 1
2,x,y\right)\leq C_H \min_{(x,y)\in\overline{\Omega _M} }\tilde n
(t,x,y).
\end{equation}
Notice that $\Omega _R \subset \Omega _M$ so that, taking $0<\nu
<{C_H}^{-1}e^{-\frac{cR}{2(B+1)}}\Vert \tilde n(\frac 1
2,\cdot,\cdot)\Vert _{L^\infty(\Omega _R)}$, we get $\nu \Gamma
_R(x,y) < e^{\frac{cx}2}\tilde n (1,x,y)=u(1,x,y)$ for all
$(x,y)\in \overline{\Omega _R}$. Now assume that there is a time
$t_0 >1$, which we assume to be the smallest one, such that
\begin{equation}\label{contact}
\nu \Gamma _R(x_0,y_0)=e^{\frac{cx_0}2}\tilde
n(t_0,x_0,y_0)=u(t_0,x_0,y_0) \quad \text{ for some } (x_0,y_0)\in
\Omega _R,
\end{equation}
and derive a contradiction, which shall conclude the proof. Thanks to the definition of $t_0$,
$\tilde n(t,x,y)-e^{-\frac{cx}2}\nu \Gamma _R(x,y)$ restricted to
$(1,t_0]\times\Omega _R$ has a zero minimum value at
$(t_0,x_0,y_0)$. The maximum principle then yields
\begin{equation}\label{eq:contact_condition}
(\partial _t-c\partial _x -\partial _{xx}-\partial
_{yy})\left[\tilde n(t,x,y)-e^{-\frac{cx}2}\nu \Gamma
_R(x,y)\right](t_0,x_0,y_0)\leq 0.
\end{equation}
Combining equation \eqref{eq-tilde-n} for $\tilde n$, equation
\eqref{vp-pb-rectangle} for the principal eigenfunction $\Gamma
_R$ and the contact condition \eqref{eq:contact_condition}, we arrive at
$$
\tilde n(t_0,x_0,y_0)\left(-\int _\R K(t_0,x_0+ct_0,y_0,y')\tilde
n(t_0,x_0,y')\,dy'-\frac{c^2}4-\lambda _R\right)\leq 0,
$$
which, in view of \eqref{inegalite} and \eqref{controle-loin}
implies
\begin{equation}\label{to-be-contradicted}
\ep <\int _{|y'|\leq M} K(t_0,x_0+ct_0,y_0,y')\tilde
n(t_0,x_0,y')\,dy'\leq k^+ 2M \Vert \tilde n(t_0,\cdot,\cdot)\Vert
_{L^\infty(\R^2)}.
\end{equation}

Now, observe that
$$
\phi(t):=e^{r_{max} t }\|n(t_0- 1/ 2,\cdot,\cdot)\|_{L^\infty(\R
^2)}
$$
solves
$\partial_t\phi-\partial_{xx}\phi-\partial_{yy}\phi-r_{max}\phi=0$
on $[t_0-1/2,t_0]\times \mathbb R^2$, whereas $n(t,x,y)$ satisfies
$\partial_t n-\partial_{xx}n-\partial_{yy}n-r_{max}n\leq 0$. Since
$\phi(0)\geq n(t_0- 1/2,x,y)$ the parabolic comparison principle
implies
$$
n(t_0,x,y)\leq \phi(1/2)=e^{\frac 12 r_{max}}\|n(t_0- 1/
2,\cdot,\cdot)\|_{L^\infty(\R ^2)}.
$$
Noticing that $\|n(t,\cdot,\cdot)\|_{L^\infty(\R ^2)}=\|\tilde
n(t,\cdot,\cdot)\|_{L^\infty(\R ^2)}$, the above inequality
combined with \eqref{to-be-contradicted} implies
$$
\ep<k^+ 2M e^{\frac 12 r_{max}} \Vert \tilde
n(t_0-1/2,\cdot,\cdot)\Vert _{L^\infty(\R^2)},
$$
and therefore
\begin{eqnarray*}
\ep &<& k^+ 2M e^{\frac 12 r_{max}}
\max\left(\max_{(x,y)\in \overline{\Omega _M}}\tilde n(t_0-1/2,x,y),\sup_{(x,y)\notin \Omega _M}\tilde n(t_0-1/2,x,y)\right)\nonumber\\
&\leq&  k^+ 2M e^{\frac 12 r_{max}} \max\left(C_H\min_{(x,y)\in
\overline{\Omega _M}}\tilde n(t_0,x,y),\sup_{|x|\geq \frac M{B+1}
\text{or} |y|\geq M}Ce^{-\mu(|x|+|y|)}\right)\nonumber
\end{eqnarray*}
where we have used the Harnack estimate \eqref{harnack} and the
control of the tails \eqref{tails}. Using \eqref{contact} we end
up with
$$
\ep< k^+ 2M e^{\frac 12 r_{max}} \max\left(C_H e^{\frac
{cR}{2(B+1)}}\nu,Ce^{-\mu\frac M{B+1}}\right)
$$
which is a contradiction, provided we select $M>0$ large enough, and then $\nu>0$ small enough.
\end{proof}

\section{The environmental gradient case}\label{s:propagation}

In this section, we consider a growth function $r(x,y)$ satisfying
Assumption \ref{hyp-unconfined}, i.e. the unconfined case. Since
$r(x,y)=\bar r(y-Bx)$, the equation \eqref{eq_n} under
consideration is then written as
\begin{equation}\label{eq_n-invasion}
 \left\{\begin{array}{lll}
          \partial_t n(t,x,y)-\partial_{xx} n(t,x,y)-\partial_{yy} n(t,x,y)\vspace{3pt}\\
      \quad\quad =\displaystyle \left(\bar r(y-B(x-ct))-\int_{\mathbb
R}K(t,x,y,y')n(t,x,y')\,dy'\right)n(t,x,y)\vspace{5pt}\\
      n(0,x,y)=n_0(x,y).
         \end{array}
\right.
\end{equation}

We define
\begin{equation}\label{critical-speed-unconfined}
c^{**}:=\begin{cases}2\sqrt{-\lambda_\infty\frac{1+B^2}{B^2}} &
\text{ if }
\lambda_\infty< 0\vspace{3pt}\\
-\infty & \text{ if } \lambda_\infty \geq0
\end{cases}
\qquad \text{the critical speed in the unconfined case,}
\end{equation}
where $\lambda _\infty$ is the principal eigenvalue defined by
\eqref{vp-pb2}. In the whole section, we are then equipped with
$\lambda _\infty$, $\Gamma _\infty ^{1D}(z)$ satisfying \eqref{vp-pb2}, and $\lambda _\infty$, $\Gamma _\infty(x,y):=\Gamma _\infty ^{1D}(z)(y-Bx)$ satisfying
\eqref{vp-pb}.

\begin{rem} Recall that in the case where $r(x,y)=\bar r(y-Bx)=1-A(y-Bx)^2$,
 $A>0$, we have $\lambda_\infty=\sqrt{A(B^2+1)}-1$, and $\Gamma_\infty(x,y)=\exp\left(-\sqrt{\frac A{B^2+1}}\frac{(y-Bx)^2}2\right)$.
 Then $A(B^2+1)\geq 1$ implies $c^{**}=-\infty$, while  $A(B^2+1)<1$ implies
\begin{equation}\label{formula-critical-speed}
 c^{**}=2\sqrt{\left(1-\sqrt{A(B^2+1)}\right)\frac{1+B^2}{B^2}}.
\end{equation}
\end{rem}

\subsection{Invasion}\label{ss:invasion-unconfined}

For slow climate shifts $0\leq c< c^{**}$, we prove that the
population survives, and indeed propagate. We define
$$
\omega_x^-=-\sqrt{-\frac{4\lambda _\infty}{1+B^2}
-\frac{B^2}{(1+B^2)^2}c^2}+\frac{B^2}{1+B^2}c\,,
$$
$$
\omega_x^+=\sqrt{-\frac{4\lambda _\infty}{1+B^2}
-\frac{B^2}{(1+B^2)^2}c^2}+\frac{B^2}{1+B^2}c\,,
$$
which are the propagation speeds in space $x$ towards $-\infty$,
$+\infty$ respectively, in a sense to be made precise in the following theorem:

\begin{theo}[Survival and invasion]\label{th:invasion}

Assume that $r\in L^\infty_{loc}(\R^2)$ and $K\in
L^\infty((0,\infty)\times\R^3)$ satisfy \eqref{Assumption_r} and
\eqref{Assumption_K} respectively. Let Assumption
\ref{hyp-unconfined} hold. Assume that $n_0\not\equiv 0$ satisfies
\eqref{initial-data}. If $0\leq c < c^{**}$,  then there is a
function $\psi:\R\to\R$ with $\psi(+\infty)=0$ such that any
nonnegative solution $n$ of \eqref{eq_n-invasion} satisfies
\begin{equation}\label{par-dessus1}
\forall (t,x)\in [1,\infty)\times \R,\quad
\max\left(\|n(t,x,\cdot)\|_\infty,\int _\R n(t,x,y)\,dy\right)\leq
\psi(x-\omega_x^+ t),
\end{equation}
\begin{equation}\label{par-dessus2}
\forall (t,x)\in [1,\infty)\times \R,\quad
\max\left(\|n(t,x,\cdot)\|_\infty,\int _\R n(t,x,y)\,dy\right)\leq
\psi(-(x-\omega_x^- t)).
\end{equation}
Moreover, for any $0<\delta<\frac{\omega_x^+-\omega_x^-}2$, there
exists $\beta>0$ such that any nonnegative solution $n$ of
\eqref{eq_n-invasion} satisfies
\begin{equation}\label{minoration-n}
\forall t\in [1,\infty),\quad \int _\R n(t,\omega_x t,y)\,dy\geq
\beta.
\end{equation}
for any $\omega_x\in [\omega_x^-+\delta,\omega_x^+-\delta]$.
\end{theo}

\begin{rem}\label{supportprop}
As it can be seen in the proof, the population will follow the
optimal trait. Also the propagation speeds of the population along
the phenotypic trait $y$ towards $-\infty$, $+\infty$ are respectively
$$
\omega_y^-=B\omega _x^--Bc,\quad \omega_y^+=B\omega _x^+-Bc.
$$
\end{rem}

\begin{rem}\label{rem:migration-evolution}
In our unconfined model \eqref{eq_n-invasion}, the conditions are shifted by the climate  at a speed $c\geq 0$ towards large $x\in\mathbb R$. From Theorem \ref{th:invasion}, the population is able to follow the climate shift only if 
\[\omega_x^-\leq c\leq \omega_x^+.\]
One can check that the first inequality is always satisfied, while the second one is only satisfied if $c\leq 2\sqrt{-\lambda_\infty}<c^{**}$. Hence, if $c\in(2\sqrt{-\lambda_\infty},c^{**})$, the population survives despite its inability to follow the climate change: it only survives because it is also able to evolve to become adapted to the changing climate. Finally, one can notice that the threshold speed $2\sqrt{-\lambda_\infty}$ is similar to the definition of the critical speed $c^*$ in the confined case (see \eqref{critical-speed}), which makes sense, since in the confined case, the survival is only possible if the population succeeds to strictly  follow the climate change.
\end{rem}

\begin{proof}
Rather than working in the $(x,y)$ variables, let us write
$n(t,x,y)=v(t,X,Y)$ where $X$ (resp. $Y$) represents the direction
of (resp. the direction orthogonal to) the optimal trait $y=Bx$,
that is
\begin{equation}\label{cgtvar}
 X=\frac {x+By}{\rac},\quad Y=\frac{-Bx+y}{\rac}.
\end{equation}
In these new variables, equation \eqref{eq_n-invasion} is  recast
as
\begin{align}
& \partial_t v-\partial_{XX} v-\partial_{YY} v\nonumber\\
&\quad =\left(\bar r\left(\rac \,Y+Bct\right)-\int_{\mathbb R}v\left(t,\frac{
\frac{X-BY}\rac
+By'}\rac,\frac{-B\frac{X-BY}\rac+y'}\rac\right)\,dy'\right)v.\label{eq-v}
\end{align}
Observe that for ease of writing we have taken $K\equiv 1$, which
is harmless since $0<k^-\leq K \leq k^+$. Note also that defining
$\Gamma(Y):=\Gamma _\infty ^{1D}(\rac\,Y)$, we have
\begin{equation}\label{eigenfunction}
-\partial_{YY}\Gamma -\bar r \left(\rac\,
Y\right)\Gamma=\lambda_\infty \Gamma.
\end{equation}

\noindent {\bf The controls from above \eqref{par-dessus1} and
\eqref{par-dessus2}.} To prove \eqref{par-dessus1}, we are seeking
for a solution of
\begin{equation}\label{sol-lin}
\partial _t \psi-\partial _{XX} \psi -\partial_{YY} \psi-\bar
r\left(\rac \,Y+Bct\right) \psi=0,
\end{equation}
in the form
\begin{equation}\label{sursol}
 \psi(t,X,Y):=e^{-\lambda(X-\omega _X^+ t)}e^{-\nu(Y-\omega _Y^+
t)}\Gamma \left(Y-\omega _Y^+ t\right),
\end{equation}
with $\lambda >0$, $\nu >0$, $\omega_X^+>0$, $\omega _Y^+>0$. If
we choose $\lambda:=\omega _X^+ /2$ and $\nu:=\omega _Y^+/2$, then
\eqref{sol-lin} turns out to be equivalent to
\begin{equation}\label{eqGamma1}
\left(\frac{{\omega _X^+}^2}4+\frac{{\omega
_Y^+}^2}4\right)\Gamma \left(Y-\omega _Y^+
t\right)-\Gamma
_{YY}\left(Y-\omega _Y^+ t\right)-\bar r\left(\rac\, Y+Bct\right)\Gamma \left(Y-\omega _Y^+ t\right)=0.
\end{equation}
The combination of \eqref{eigenfunction} and \eqref{eqGamma1}
shows that $\psi$ is a solution of \eqref{sol-lin} if we select
\begin{equation}\label{speeds-invasion}
\omega_X ^+:=\sqrt{-4\lambda _\infty
-\frac{B^2}{1+B^2}c^2}=\sqrt{\frac{B^2}{1+B^2}\left({c^{**}}^2-c^2\right)},\quad
\omega _Y ^+:=\frac {-B}{\sqrt{1+B^2}}\,c.
\end{equation}

Now, since $v(0,\cdot,\cdot)$ is compactly supported we can choose
$M>0$ large enough so that $M\,\psi(0,X,Y)\geq v(0,X,Y)$ for all $(X,Y)\in\mathbb R^2$. In view
of \eqref{eq-v}, $\partial _t v-\partial _{XX} v -\partial_{YY}
v-\bar r\left(\rac \,Y+Bct\right) v\leq 0$, so that the parabolic
comparison principle yields
\begin{equation}\label{control-above}
v(t,X,Y)\leq M\psi(t,X,Y)\leq M  e^{-\frac{\omega _X^+}2(X-\omega
_X ^+t)}e^{-\frac{\omega _Y ^+}2(Y-\omega _Y^+ t)}.
\end{equation}
Noticing that
\begin{equation}\label{expressions}
\omega_x^+=\frac{\omega_X^+-B\omega_Y^+}{\rac},\quad
\omega_y^+=\frac{B\omega_X^++\omega_Y^+}{\rac},
\end{equation}
and going back to the original variables, we arrive at
\begin{equation}\label{control-above2}
n(t,x,y)\leq  M  e^{-\frac{\omega _x^ +}2(x-\omega _x
^+t)}e^{-\frac{\omega _y ^+}2(y-\omega _y^+ t)}=Me^{-\frac
12\sqrt{1+B^2}\omega _X^+(x-\omega _x^+t)}e^{-\frac 12\omega _y
^+(y-B(x-ct))},
\end{equation}
where we have used the relation $\omega_y^+=B\omega _x^+-Bc$
together with expressions \eqref{expressions} (notice also that
${\omega_x^+}^2+{\omega _y ^+}^2={\omega _X^+}^2+{\omega
_Y^+}^2$). Combining the control above  with the control of the
tails \eqref{tails}, one obtains \eqref{par-dessus1}.

Finally, by using
$$
 \psi(t,X,Y):=e^{\frac{\omega _X ^+}2(X+\omega _X^+ t)}e^{-\frac{\omega _Y ^+}2(Y-\omega _Y^+
t)}\Gamma \left(Y-\omega _Y^+ t\right),
$$
rather than \eqref{sursol} and using similar arguments, we prove
\eqref{par-dessus2}, remarking that
\begin{equation}\label{expressions-gauche}
\omega_x^-=\frac{-\omega_X^+-B\omega_Y^+}{\rac},\quad
\omega_y^-=\frac{-B\omega_X^++\omega_Y^+}{\rac}.
\end{equation}

\noindent {\bf The control from below \eqref{minoration-n}.} The
first step is to estimate the nonlocal term in \eqref{eq-v}. Let
$(t_0,X_0,Y_0)\in[1,\infty)\times \R^2$ be given, and the
corresponding $(x_0,y_0)$ obtained through the change of variable
\eqref{cgtvar}. We select $M>2$ such that $|y_0-B(x_0-ct_0)|\leq
M$. The control of the tails \eqref{tails} then implies
\begin{align}
& \int_{\mathbb R}v\left(t_0,\frac{ \frac{X_0-BY_0}\rac
+By'}\rac,\frac{-B\frac{X_0-BY_0}\rac+y'}\rac\right)\,dy'=\int_ \R n(t_0,x_0,y')\,dy'\nonumber\\
&\quad \leq
2M\max_{y\in[-M,M]}n\left(t_0,x_0,B(x_0-ct_0)+y\right)
+\int_{[-M,M]^c}Ce^{-\mu |y|}\,dy.\label{bidule}
\end{align}

In order to estimate the first term of the above expression, let us recall that the uniform boundedness of the solutions is known since Lemma \ref{lem:tails}. This allows to use the refinement of Harnack inequality, namely 
 Theorem~\ref{th:harnack-mod} with $\omega=\mathbb R\times \{0\}$ and $\delta=\frac{C}{2M\mu}e^{-\mu M}>0$ to $\tilde v(t,X,Y)=v\left(t,X,Y-\frac{Bc}{\sqrt{1+B^2}}t\right)$. $\tilde v$ indeed satisfies
 \begin{align*}
& \partial_t \tilde v+\frac{Bc}{\sqrt{1+B^2}}\partial_Y\tilde v-\partial_{XX} \tilde v-\partial_{YY} \tilde v\nonumber\\
&\quad =\left(\bar r\left(\rac \,Y\right)-\int_{\mathbb R} v\left(t,\frac{
\frac{X-BY}\rac
+By'}\rac,\frac{-B\frac{X-BY}\rac+y'}\rac-\frac{Bc}{\sqrt{1+B^2}}t\right)\,dy'\right)\tilde v,
\end{align*}
and there exists thus a constant $\tilde C_M>0$, depending on $M$, such that
\begin{align}
&\max_{(x,y)\in[-M, M]^2}n\left(t_0,x_0+x,B(x_0-ct_0)+y\right)\nonumber\\
&\qquad\leq \tilde C_M\min_{(x,y)\in[- M, M]^2}n\left(t_0,x_0+x,B(x_0-ct_0)+y\right)+\delta\nonumber\\
&\qquad\leq \tilde C_M n(t_0,x_0,y_0)+\delta,\label{bidule2}
\end{align}
which we plug into \eqref{bidule} to get
\begin{equation}\label{est-nonlocal}
 \int_{\mathbb R}v\left(t_0,\frac{
\frac{X_0-BY_0}\rac
+By'}\rac,\frac{-B\frac{X_0-BY_0}\rac+y'}\rac\right)\,dy'\leq
C_M v(t_0,X_0,Y_0))+\frac {3C} \mu
e^{-\mu M}.
\end{equation}

\medskip

Next, for $R>0$, let us consider $\tilde \lambda _R$, $\tilde \Gamma _R(Y)$ solving the \lq\lq one dimensional'' principal eigenvalue problem
\begin{equation}\label{pb-vp-Y}
\begin{cases}
-\partial_{YY}\tilde \Gamma_R-\bar
r\left(\rac\, Y\right)\tilde \Gamma_R=\tilde \lambda_R\tilde \Gamma_R  \quad\text{ in  }(-R,R)\vspace{3pt}\\
\tilde \Gamma _R=0 \quad\text{ on }  \partial ((-R,R))\vspace{3pt}\\
\tilde \Gamma_R >0\quad \text{ on }  (-R,R),\quad \tilde \Gamma _R(0)=1.
 \end{cases}
 \end{equation}
 We have $\tilde \lambda_R \to
\lambda_\infty$ as $ R\to\infty$. Defining
$$
\Gamma _R(X,Y):=\cos\left(\frac x R \frac \pi 2\right)\tilde \Gamma _R(Y),
$$
we get
\begin{equation}\label{pb-vp-XY}
\begin{cases}
-\partial_{XX} \Gamma_R-\partial_{YY}\Gamma_R-\bar
r\left(\rac\, Y\right)\Gamma_R=\lambda_R\Gamma_R  \quad\text{ in  }(-R,R)^2\vspace{3pt}\\
\Gamma _R=0 \quad\text{ on }  \partial ((-R,R)^2)\vspace{3pt}\\
\Gamma_R >0\quad \text{ on }  (-R,R)^2,\quad \Gamma _R(0,0)=1,
 \end{cases}
\end{equation}
where $\lambda _R=\tilde \lambda  _R+\frac{\pi ^2}{4R^2}\to \lambda _\infty$ as $R\to\infty$.

We then define, for  some $\omega _X$ to be specified later and
$\beta>0$ to be selected later,
\begin{equation}\label{def:psi}
 \psi_\beta(t,X,Y):=\beta e^{-\frac{\omega_X}2(X-\omega_Xt)}e^{-\frac{\omega_Y^+}2(Y-\omega_Y^+ t)}\Gamma_R(X-\omega_Xt,Y-\omega_Y^+ t),
\end{equation}
which satisfies
\begin{equation}\label{truc}
 \partial_t\psi_\beta-\partial_{XX}\psi_\beta-\partial_{YY}\psi_\beta-\bar r\left(\rac Y+Bct\right)\psi_\beta=\left(\frac{\omega_X^2}4+\frac{{\omega_Y^+}^2}4+\lambda_R\right)\psi_\beta.
\end{equation}
Since $\psi_\beta(t,\cdot,\cdot)$ is compactly supported and
$v(1,\cdot,\cdot)>0$, we can assume that $\beta>0$ is small enough so that
\begin{equation}\label{psi-v}
 \psi_\beta(1,\cdot,\cdot)<v(1,\cdot,\cdot).
\end{equation}
Assume by contradiction that the set $\{t\geq 1:\,\exists
(X,Y),\,v(t,X,Y)=\psi_\beta(t,X,Y)\}$ is nonempty, and define
$$
t_0:=\min\{t\geq 1:\,\exists
(X,Y),\,v(t,X,Y)=\psi_\beta(t,X,Y)\}\in (1,\infty).
$$
Hence, $\psi_\beta -v$ has a zero maximum value at some point
$(t_0,X_0,Y_0)$ which satisfies $t_0>1$ and $(X_0+\omega_X t_0,Y_0+\omega_Y^+t_0)\in(-R,R)^2$, since
$v(t_0,\cdot,\cdot)>0$. This implies that
$$
\left[\partial_t(\psi_\beta-v)-\partial_{XX}(\psi_\beta-v)-\partial_{YY}(\psi_\beta-v)-\bar
r\left(\rac\,
Y_0+Bct_0\right)(\psi_\beta-v)\right](t_0,X_0,Y_0)\geq 0.
$$
In view of the equation \eqref{eq-v} for $v$ and the equation
\eqref{truc} for $\psi_\beta$, we infer that
$$
\left[\frac{\omega_X^2}4+\frac{{\omega_Y^+}^2}4+
\lambda_R+\int_{\mathbb R}v\left(t_0,\frac{ \frac{X_0-BY_0}\rac
+By'}\rac,\frac{B\frac{X_0-BY_0}\rac+y'}\rac\right)\,dy'\right]
\psi_\beta(t_0,X_0,Y_0)\geq 0.
$$
Hence, using $\psi_\beta(t_0,X_0,Y_0)>0$ and estimate
\eqref{est-nonlocal}, we end up with
\begin{eqnarray}
0&\leq& \frac{\omega_X^2}4+\frac{{\omega_Y^+}^2}4+\lambda_R+C_M \psi _\beta(t_0,X_0,Y_0)+\frac {3C} \mu e^{-\mu M}\nonumber\\
&\leq&
\lambda_R-\lambda_\infty+\frac{\omega_X^2-{\omega_X^+}^2}4+C_M
\psi _\beta(t_0,X_0,Y_0)+\frac {3C}
\mu e^{-\mu M},\label{ineqlowerbound}
\end{eqnarray}
thanks to \eqref{speeds-invasion} and \eqref{def:psi}.

Now, let $\delta\in\left(0,\omega_X^+\right)$ be given. For all
$\omega _X$ --- appearing in \eqref{def:psi}--- such that
$|\omega_X|\leq \omega_X^+-\delta$, we have
\begin{equation}\label{coercivite}
 \frac{\omega_X^2-{\omega_X^+}^2}4\leq -\delta\frac{2\omega_X^+-\delta}4<0.
\end{equation}
Since $\lambda_R \to
\lambda_\infty$, we can successively select $R>0$, $M>2$ large enough and $\beta >0$ small enough so that
$$
|\lambda_R-\lambda_\infty|+\frac {3C} \mu e^{-\mu
M}+C_M \psi
_\beta(t_0,X_0,Y_0)\leq \frac{2\omega_X^+-\delta}8\delta.
$$
This estimate and \eqref{coercivite} show that
\eqref{ineqlowerbound} then leads to a contradiction.

As a result, we have shown that, for any
$\delta\in\left(0,\omega_X^+\right)=\left(0,\rac\frac{\omega
_x^+-\omega _x^-}2\right)$, there are $R>0$ and $\beta>0$ such that, for
any $|\omega_X|\leq \omega_X^+-\delta$,
$$
v(t,X,Y)\geq \beta
e^{-\frac{\omega_X}2(X-\omega_Xt)}e^{-\frac{\omega_Y^+}2(Y-\omega_Y^+
t)}\Gamma_R(X-\omega_Xt,Y-\omega_Y^+ t),
$$
for all $t\geq 1$, $X\in \R$, $Y\in \R$. Defining
$$
\omega_x:=\frac{\omega_X-B\omega_Y^+}{\rac},\quad
\omega_y:=\frac{B\omega_X+\omega_Y^+}{\rac}=B\omega _x-Bc,
$$
that is the analogous of expressions  \eqref{expressions}, we can
derive the analogous of \eqref{control-above2}, that is
$$
n(t,x,y)\geq \beta e^{-\frac{\omega _x}2(x-\omega
_xt)}e^{-\frac{\omega _y}2(y-\omega _yt)}\Gamma
_R\left(\frac{x+By}{\sqrt{1+B^2}}-\omega
_Xt,\frac{-Bx+y}{\sqrt{1+B^2}}-\omega _Y ^+ t\right).
$$
This in turn implies
$$
n(t,\omega _x t,y+\omega _y t)\geq \beta\, e^{-\frac{\omega
_y}2y}\Gamma _R\left(\frac B{\rac}\, y, \frac 1{\rac}\, y\right),
$$
which holds for any $\omega_x\in
\left[\frac{-(\omega_X^+-\delta)-B\omega_Y^+}{\rac},\frac{(\omega_X^+-\delta)-B\omega_Y^+}{\rac}\right]=
\left[\omega_x^-+\frac\delta\rac,\omega_x^+-\frac\delta\rac\right]$.
This estimate is then enough to prove
\eqref{minoration-n}.\end{proof}

\subsection{Extinction}

We state our result of extinction of the population for rapid
climate shifts $c> c^{**}$.
\begin{prop}[Extinction]\label{prop:extinction-unconfined}
Assume that $r\in L^\infty_{loc}(\R^2)$ and $K\in
L^\infty((0,\infty)\times\R^3)$ satisfy \eqref{Assumption_r} and
\eqref{Assumption_K} respectively. Let Assumption
\ref{hyp-unconfined} hold. Assume that $n_0$ is compactly
supported. If $c> c^{**}$,  then any nonnegative solution $n$ of
\eqref{eq_n-invasion} satisfies, for some $\gamma _0>0$,
\begin{equation}\label{machin}
\sup _{x\in \R}\,\int _\R n(t,x,y)\,dy=\mathcal O\left(e^{-\gamma
_0 t}\right)\to 0,\, \text{ as } t\to \infty.
\end{equation}
\end{prop}

\begin{proof}
Just as in the first part of the proof of Theorem
\ref{th:invasion}, we introduce the change of variables
\eqref{cgtvar} and $v$ satisfying \eqref{eq-v}. We are seeking for
a solution $\varphi$ of \eqref{sol-lin} in the form
\begin{equation*}\label{sursol2}
 \varphi(t,X,Y):=e^{-\gamma t}e^{-\nu(Y-\omega _Y^+
t)}\Gamma \left(Y-\omega _Y^+ t\right),
\end{equation*}
with $\nu >0$, $\omega _Y^+>0$. If we choose $\nu:=\omega _Y^+/2$,
then $\varphi$ is a solution of \eqref{sol-lin} if and only if
\begin{equation}\label{eqGamma3}
\left(\frac{{\omega _Y^+}^2}4-\gamma\right)\Gamma
\left(Y-\omega _Y^+ t\right)-\Gamma
_{YY}\left(Y-\omega _Y^+ t\right)-\bar r\left(\rac\, Y+Bct\right)\Gamma \left(Y-\omega
_Y^+ t\right)=0.
\end{equation}
The combination of \eqref{eigenfunction} and \eqref{eqGamma3}
shows that $\varphi$ is a solution of \eqref{sol-lin} if we select
$\gamma={\omega_Y^+}^2+\lambda_\infty>0$. Since $v(0,\cdot,\cdot)$
is compactly supported we can choose $M>0$ large enough so that
$M\varphi(0,X,Y)\geq v(0,X,Y)$. In view of \eqref{eq-v}, $\partial
_t v-\partial _{XX} v -\partial_{YY} v-\bar r\left(\rac
\,Y+Bct\right) v\leq 0$ so that the parabolic comparison principle
yields
\begin{equation}\label{control-above3}
v(t,X,Y)\leq M\varphi(t,X,Y)\leq M e^{-\gamma t}e^{-\mu(Y-\omega
_Y^+ t)}\Gamma \left(Y-\omega _Y^+ t\right)\leq M e^{-\gamma
t}e^{-\mu(Y-\omega _Y^+ t)}.
\end{equation}
Then, if w.l.o.g. $c\geq 0$, we get, for $\alpha>0$ to be chosen
later,
\begin{eqnarray}
 \sup_{x\in\mathbb R}\int_{\R}n(t,x,y)\,dy&=& \sup_{x\in\mathbb R}\left(\int_{y\geq B(x-ct)-\alpha t}n(t,x,y)\,dy+\int_{y\leq B(x-ct)-\alpha t}n(t,x,y)\,dy\right).\label{est0}
\end{eqnarray}
We can estimate the first term of \eqref{est0} as follows
\begin{eqnarray*}
\int_{y\geq B(x-ct)-\alpha t}n(t,x,y)\,dy &\leq & \int_{y\geq B(x-ct)-\alpha t}v\left(t,\frac{x+By}{\rac},\frac{-Bx+y}{\rac}\right)\,dy\\
 &\leq& M\int_{y\geq B(x-ct)-\alpha t}e^{-\gamma t}e^{-\mu\left(\frac{-Bx+y}{\rac}-\omega _Y^+
t\right)}\,dy\\
&\leq& M \int_{y\geq B(x-ct)-\alpha t}e^{-\gamma t}e^{-\frac\mu{\rac}\left(y-B(x-ct)\right)}\,dy\\
&\leq& M \int_{\R_+}e^{-\gamma t}e^{-\frac\mu{\rac}(\tilde
y-\alpha t)}\,d\tilde y\leq M\frac\rac\mu e^{\frac{-\gamma}2t},
\end{eqnarray*}
if we choose $\alpha=\frac{\gamma\rac}{2\mu}$. Using the control
of the tails \eqref{tails}, we can estimate the second term of
\eqref{est0} by
\begin{eqnarray*}
\int_{y\leq B(x-ct)-\alpha t}n(t,x,y)\,dy&\leq&\int_{y\leq B(x-ct)-\alpha t}Ce^{-\mu|y-B(x-ct)|}\,dy\\
&\leq&\int_{\R_+}Ce^{-\mu\left(\tilde y+\alpha t\right)}\,d\tilde
y\leq \frac{C}\mu e^{-\mu\alpha t}.
\end{eqnarray*}
Then, \eqref{est0} becomes
$$
\sup_{x\in\mathbb R}\int_{\R}n(t,x,y)\,dy\leq
Ce^{-\min\left(\frac{\gamma}2,\mu\alpha\right)t},
$$
which proves the proposition.
\end{proof}

\section{Mixed scenarios}\label{s:mixed}

In this section, we consider a growth function $r(x,y)$ satisfying
Assumption \ref{hyp-mixed}, that is
\begin{equation}\label{r=rurc}
r(x,y)=\mathbf 1_{\R_-\times \R}(x,y)r_u(x,y)+\mathbf
1_{\R_+\times \R}(x,y)r_c(x,y),
\end{equation}
 where $r_c(x,y)$ satisfies Assumption \ref{hyp-confined} and $r_u(x,y)=\bar r_u(y-Bx)$
 satisfies Assumption \ref{hyp-unconfined}.  It follows from subsection \ref{subsection:eigenvalue_pb} --- see \eqref{vp-pb} and \eqref{vp-pb2}--- that we can define the
  principal eigenvalues $\lambda _\infty$, $\lambda _{u,\infty}$, and some principal eigenfunctions $\Gamma _\infty(x,y)$,
   $\Gamma _{u,\infty}(x,y)=\Gamma _{u,\infty}^{1D}(y-Bx)$ associated to $r$, $r_u$ respectively. In the sequel, for $\theta>0$, we shall
 use the following modified growth
functions,
\begin{equation}\label{def:rA}
r^\theta (x,y):=\max(r(x,y),-\theta),\quad
r_u^\theta(x,y):=\max(r_u(x,y),-\theta).
\end{equation}

We  also define the principal eigenvalues $\lambda^\theta_\infty$,
$\lambda^\theta_{u,\infty}$ and some principal eigenfunctions
$\Gamma^\theta_\infty(x,y)$, $\Gamma^\theta_{u,\infty}(x,y)=\Gamma
^{\theta,1D}_{u,\infty}(y-Bx)$ associated to $r^\theta$,
$r_u^\theta$ respectively. Using the analogous of \eqref{vp-pb2}
satisfied by $\Gamma _{u,\infty}^{1D}(z)$ and $\Gamma
_{u,\infty}^{\theta,1D}(z)$ and using the same arguments as those
used to prove \eqref{claim-ne}, we see that
\begin{equation}\label{cv-lambda}
\lambda ^\theta _{u,\infty} \nearrow \lambda
_{u,\infty}, \quad\text{ as }\theta \to \infty.
\end{equation}

\subsection{More preliminary results on the tails}

Let us first provide some estimates on the tails of the four
principal eigenfunctions defined above.

\begin{lem}[Tails of eigenfunctions from above]\label{lem:tail_Gamma_m} Assume that
$r\in L^\infty_{loc}(\R^2)$ satisfies \eqref{Assumption_r} and
Assumption \ref{hyp-mixed}. Then, for any $\mu>0$, there is
$C_\mu>0$ such that,  for all $(x,y)\in\R^2$,
\begin{equation}\label{1}
 \Gamma
_{u,\infty}(x,y)\leq C_\mu e^{-\mu |y-Bx|}, \quad
\Gamma_\infty(x,y)\leq C_\mu
e^{-\mu\max\left(|y-Bx|,x\right)}.\end{equation}
 On
the other hand, for any $\theta>2\max(\lambda _{u,\infty},\lambda
_\infty,0)$, there is $C_\theta>0$ such that, for all $(x,y)\in\R^2$,
\begin{equation}\label{2}
\Gamma_{u,\infty}^\theta(x,y)\leq C_\theta e^{-\sqrt{\frac{\theta}{2(1+B^2)}}
|y-Bx|}, \quad \Gamma_\infty^\theta(x,y)\leq C_\theta e^{-\sqrt{\frac{\theta}{2(1+B^2)}}
 \max(|y-Bx|,x)}.
\end{equation}
\end{lem}

\begin{proof} Notice that a similar estimate for the confined
case was obtained in \eqref{bornegammainf}: the proof consisted in combining the fact that $r_c(x,y)\to -\infty$ as $|x|+|y|\to \infty$ with the elliptic comparison principle. Using
$r_u(x,y)\to -\infty$ as $|y-Bx|\to \infty$, and $r(x,y) \to
-\infty$ as $\max(|y-Bx|,x)\to\infty$, we obtain \eqref{1} in a
similar manner.

As far as \eqref{2} in concerned, let us only notice that
$r_u^\theta (x,y)\geq r_u(x,y)$, $r^\theta(x,y)\geq r(x,y)$, so
that $\lambda _{u,\infty}^\theta \leq \lambda _{u,\infty}$,
$\lambda _\infty ^\theta \leq \lambda _\infty$ and therefore
$\lambda _{u,\infty}^\theta -\theta \leq -\frac \theta 2$,
$\lambda _{\infty}^\theta -\theta \leq -\frac \theta 2$. These
inequalities are valid \lq\lq far away'' (in a appropriate sense
with respect to the considered case) and enable us to
reproduce again the argument in \eqref{bornegammainf}. Details are omitted.
\end{proof}

Next, we estimate from below the tails of the principal
eigenfunction $\Gamma _{u,\infty}^\theta(x,y)$.

\begin{lem}[Tails of the eigenfunction from below]\label{lem:modified_tails}
Assume that $r\in L^\infty_{loc}(\R^2)$ satisfies
\eqref{Assumption_r} and Assumption \ref{hyp-mixed}. Then there is $\theta_1>0$ such that, for any $\theta>\theta _1$, there is $C_\theta>0$ and such that, for all $(x,y)\in\R^2$,
\begin{equation}\label{lowerbound}
\Gamma_{u,\infty}^\theta(x,y)\geq C _\theta e^{-2\sqrt{\frac{2\theta}{1+B^2}} |y-Bx|}.
\end{equation}
\end{lem}

\begin{proof} Since $r_u$ satisfies Assumption \ref{hyp-unconfined}, there
exists $R>0$ such that $r_u(x,y)=\bar
r_u(y-Bx)\leq-\theta=r_u^\theta(x,y)$ as soon as $|y-Bx|\geq R$.
Moreover, $\Gamma _{u,\infty}^\theta$ only depending on $y-Bx$, there is $\delta
>0$ such that $\Gamma _{u,\infty}^\theta(x,y)\geq \delta$ for
$(x,y)$ such that $\vert y-Bx\vert =R$. Moreover since, for
$(x,y)$ such that $|y-Bx|>R$,
$$-\partial_{xx}\Gamma_{u,\infty}^\theta(x,y)-\partial_{yy}\Gamma_{u,\infty}^\theta=(\lambda_{u,\infty}^\theta-\theta)\Gamma_{u,\infty}^\theta(x,y)\geq -2\theta \Gamma_{u,\infty}^\theta(x,y),$$
for any $\theta>\theta _1$, with $\theta _1>0$ large enough. It therefore follows from the elliptic comparison principle that
$$
\Gamma_{u,\infty}^\theta(x,y)\geq \delta e^{-2\sqrt{\frac{2\theta}{1+B^2}}(|y-Bx|-R)},
$$
which concludes the proof.
\end{proof}

Finally, we provide a control of the tails of the solution of the Cauchy problem \eqref{eq_n}, which
is a extension (and an improvement) of Lemma~\ref{lem:tails} to the mixed case.

\begin{lem}[Exponential decay of tails of $n(t,x,y)$]\label{lem:tails2} Assume that $r\in L^\infty_{loc}(\R^2)$ and $K\in
L^\infty((0,\infty)\times\R^3)$ satisfy \eqref{Assumption_r} and
\eqref{Assumption_K} respectively. Let Assumption \ref{hyp-mixed}
hold.  Assume that $n_0$ satisfies
\eqref{initial-data}. Then, for any
$\mu>0$, there is $C>0$  such that, for any global nonnegative
solution $n$ of \eqref{eq_n},
\begin{equation}\label{tails-mixed} 0\leq n(t,x,y)\leq Ce^{-\mu\max\left(|y-B(x-ct)|,x-ct\right)},
\end{equation}
for all $t\geq 0$, $x\in \R$, $y \in \R$.
\end{lem}

\begin{proof} By using arguments similar to those of Lemma~\ref{lem:tails} (for the bounded and unbounded cases), we see that there are $\bar \mu>0$ and $\bar C>0$ such that
\begin{equation}\label{est3}
 n(t,x,y)\leq \bar Ce^{-\bar \mu\max\left(|y-B(x-ct)|,x-ct\right)},
\end{equation}
for all $t\geq 0$, $x\in \R$, $y \in \R$. Details are omitted.

Next, let $\mu>0$ be given, and $\nu>0$ to be determined later. Thanks to
Assumption~\ref{hyp-mixed}, there exists $R>0$ such that
$\min\left(|y-B(x-ct)|,x-ct\right)>R$ implies $r(x,y)<-\nu$.
Similarly to the proof of Lemma~\ref{lem:tails}, we introduce
$\varphi(t,x,y):=\kappa e^{-\mu(|y-B(x-ct)|-R)}$, which
satisfies
$$
 \partial_t\varphi-\partial_{xx}\varphi-\partial_{yy}\varphi-r(x-ct,y)\varphi = \left(\pm\mu B c-\mu^2B^2-\mu^2+\nu\right)\varphi\geq 0,
 $$
for all $t\geq 0$, and $|y-B(x-ct)|>R$, provided we chose $\nu>0$
large enough. Thanks to \eqref{est3}, we can chose $\kappa$ large
enough for $n(t,x,y)\leq \varphi(t,x,y)$ to hold for all $t\geq 0$, and
$|y-B(x-ct)|=R$. Since moreover $n$ satisfies $\partial _t n-\partial _{xx}n-\partial _{yy}n-r(x-ct,y)n\leq 0$, the
parabolic maximum principle implies that $n(t,x,y)\leq
\varphi(t,x,y)$ for $t\geq 0$, and $|y-B(x-ct)|>R$. The same argument can be made with $\varphi(t,x,y):=\kappa e^{-\mu(|x|-R)}$
for $t\geq 0$ and $x\geq R$ with $R>0$ large enough, which
is enough to prove the lemma.
\end{proof}

\subsection{Extinction, survival, propagation}\label{subsection:theomixed}

Equipped with the principal eigenvalues $\lambda _\infty$,
$\lambda_{u,\infty}$, we adapt \eqref{critical-speed} and
\eqref{critical-speed-unconfined} by defining
\begin{equation*}\label{critical-speed2}
c^*:=\begin{cases}2\sqrt{-\lambda_\infty} & \text{ if }
\lambda_\infty< 0\vspace{3pt}\\
-\infty & \text{ if } \lambda_\infty \geq0,
\end{cases}
\end{equation*}
and
\begin{equation*}\label{critical-speed-unconfined2}
c^{**}_u:=\begin{cases}2\sqrt{-\lambda_{u,\infty}\frac{1+B^2}{B^2}}
& \text{ if }
\lambda_{u,\infty}< 0\vspace{3pt}\\
-\infty & \text{ if } \lambda_{u,\infty} \geq0.
\end{cases}
\end{equation*}
Since $-\lambda_{u,\infty}$ can be smaller that $-\lambda_\infty$,
it may happen that $ c_u^{**}\leq c^*$, in contrast with
 Section~\ref{s:propagation} where $c^*\leq c^{**}$ was always true. The last result of this study provides a qualitative description of the
dynamics of the population depending on the relative values of
$c^*$, $c_u^{**}$, and the speed $c$ of the climate shift.

\begin{theo}[Long time behavior in the mixed case]\label{th:mixed} Assume that $r\in L^\infty_{loc}(\R^2)$ and $K\in
L^\infty((0,\infty)\times\R^3)$ satisfy \eqref{Assumption_r} and
\eqref{Assumption_K} respectively. Let Assumption \ref{hyp-mixed}
hold.  Assume that $n_0\not\equiv 0$ satisfies
\eqref{initial-data}.  Let $n$ be a global nonnegative solution of
\eqref{eq_n}.

\begin{enumerate}[(i)]
\item Assume $\max(c^\ast,c_u^{**})<c$. Then the population gets
extinct exponentially fast. More precisely, there are $C>0$ and
$\gamma _0 >0$ such that
\begin{equation}\label{point1}
\sup_{x\in\R} \int_{\R}n(t,x,y)\,dy\leq Ce^{-\gamma_0 t}, \quad \forall t\geq 1.
\end{equation}

\item Assume $c_u^{**}<c<c^*$. Then the population survives and
follows the climate shift, but does not succeed to propagate. More precisely, there are  $\beta>0$, $C>0$ and
$\omega >0$ such that
\begin{equation}\label{point2-minoration}
\int _\R n(t,x+ct,y)\,dy\geq \beta, \quad \forall t\geq 1, \forall
x\in[-1,1],
\end{equation}
while
\begin{equation}\label{point2-majoration}
\int _\R  n(t,x+ct,y)\,dy\leq Ce^{\omega x}, \quad \forall
t\geq 1, \forall x\in \R.
\end{equation}

\item Assume $c^*<c<c_u^{**}$. Then the population survives, but
does not succeed to follow the climate shift. More precisely,
there are $\beta>0$, $C>0$, $\omega >0$, and $\gamma
>0$ such that
\begin{equation}\label{point3-minoration}
\int _\R n\left(t,x+\frac{B^2c}{1+B^2}t,y\right)\,dy\geq \beta,
\quad \forall t\geq 1, \forall x\in[-1,1],
\end{equation}
while
\begin{equation}\label{point3-majoration}
\int _\R n(t,x+ct,y)\,dy\leq Ce^{-\omega x}e^{-\gamma t}, \quad
\forall t\geq 1, \forall x\in \R.
\end{equation}

 \item Assume $c<\min(c^\ast,c_u^{**})$.  Then the population
survives with an increasing species' range. More precisely, there
is $\beta >0$ such that
 \begin{equation}\label{point4-minoration}
\min_{x\in\left[\frac{B^2c}{1+B^2}t,ct\right]}\int _\R
n\left(t,x,y\right)\,dy\geq \beta,  \quad \forall t\geq 1.
\end{equation}
\end{enumerate}

\end{theo}

\begin{rem}\label{remint}
 Notice that if $c<\min(c^*,c_u^{**})$, then the population will survive for $x\in\left[\frac{B^2c}{1+B^2}t,ct\right]$. The
  size of its range will then increase at a speed of at least $\frac c{1+B^2}$. Moreover, this speed will
   provide little information on $\max(c^\ast,c_u^{**})-c$, that is on the tolerance of the population to an increase
    of the climate change speed. The situation is then qualitatively different from the unconfined case, where the growth
     of the range of the population was directly linked to the difference
     $c^{**}-c$ (see subsection \ref{ss:invasion-unconfined} for
     details):
 \begin{eqnarray*}\omega_x^+-\omega_x^-&=&2\sqrt{-\frac{4\lambda _\infty}{1+B^2}
-\frac{B^2}{(1+B^2)^2}c^2}\\
&=&\frac {2B}{1+B^2}\sqrt{(c^{**})^2-c^2}.
\end{eqnarray*}
\end{rem}

\begin{proof}[Proof of $(i)$]
The first lines of the proof of Proposition \ref{prop:extinction} shows that
\eqref{controle1} remains valid here. Therefore, in view of \eqref{1}, for any $\mu>0$, there is $C_\mu >0$ such that, for all $(x,y)$,
\begin{equation}\label{est1}
e^{\frac{cx}2}n(t,x+ct,y)\leq
C_\mu e^{\left(-\lambda_\infty-\frac{c^2}4\right)t}e^{-\mu\max\left(|y-Bx|,x\right)}.
\end{equation}
Then, in particular,
\begin{equation}\label{est-tails}
n(t,ct,y)\leq C_\mu e^{\left(-\lambda_\infty-\frac{c^2}4\right)t}e^{-\mu|y|},\quad \textrm{ for }t\geq 0,\,y\in\R.
\end{equation}

Next, notice that $n$ satisfies
\begin{equation}\label{eq-demi-espace}
 \partial_t n-\Delta n\leq r(x-ct,y)n = r_u (x-ct,y)n\leq r_u^\theta (x-ct,y)n,\quad \textrm{ for } t\geq 0,\,x\leq ct,\,y\in\R.
\end{equation}
We now build a supersolution for \eqref{eq-demi-espace}, using an
approach similar to the one developed in the proof of Proposition
\ref{prop:extinction-unconfined}. Recall that $\omega_Y^+$ was
defined in \eqref{speeds-invasion}. Let $\theta>\max\left(2\lambda
_{u,\infty},2\lambda _\infty ,\theta
_1,\frac{B^2c^2}{64(1+B^2)},2{\omega_Y^+}^2\right)$ to be chosen
later. Define $\Gamma _u^\theta(Y):=\Gamma _u^{\theta,1D}(\rac\,
Y)$, which in turn implies $\Gamma _u^\theta(Y)=\Gamma
_{u,\infty}^{\theta}(x,y)$, with $(x,y)$ related to $(X,Y)$ by
\eqref{cgtvar}. Define
\begin{equation}\label{def:varphi}
\varphi(t,X,Y):=e^{-\gamma
t}e^{\frac{-\omega_Y^+}2(Y-\omega_Y^+t)}\Gamma_u^\theta(Y-\omega_Y^+
t),
\end{equation}
with
$\gamma>0$ to be chosen later.   Recall that $r_u^\theta (x,y)=\bar r _u ^\theta (y-Bx)$. We compute
$$
\partial _t\varphi -\Delta \varphi -\bar r _u^{\theta}\left(\rac\, Y+Bct\right)\varphi=\left(-\gamma +\frac{{\omega _Y ^+}^2}{4}+\lambda _{u,\infty} ^\theta\right)\varphi \geq 0,
$$
as soon as
\begin{equation}\label{def:bargamma}
\gamma\leq \bar\gamma:=
\frac{{\omega_Y^+}^2}4+\lambda_{u,\infty}^\theta=\frac{B^2(c^2-
(c_u^{**})^2)}{4(1+B^2)}+\left(\lambda_{u,\infty}^\theta-\lambda_{u,\infty}\right).
\end{equation}
Since $c^2- (c_u^{**})^2>0$ and
$\lim_{\theta\to\infty}\lambda_{u,\infty}^\theta=\lambda_{u,\infty}$, we have
$\bar \gamma>0$ provided we fix $\theta$ large enough. As a result
\begin{equation}\label{def:barn}
\bar
n(t,x,y):=\varphi\left(t,\frac{x+By}{\sqrt{1+B^2}},\frac{-Bx+y}{\sqrt{1+B^2}}\right)
\end{equation}
is the requested supersolution for \eqref{eq-demi-espace}, that is
\begin{equation}\label{eq:super_ru}
\partial_t \bar n-\Delta \bar n\geq r_u^\theta(x-ct,y) \bar n.
\end{equation}

Now, we take care of the line $x=ct$. Using the definition \eqref{speeds-invasion} of $\omega_Y^+$, we get
\begin{eqnarray}
 \bar n(t,ct,y)
&=&e^{-\gamma
t}e^{\frac{-\omega_Y^+}2\left(\frac{-Bct+y}{\sqrt{1+B^2}}-\omega_Y^+t\right)}\Gamma_u^\theta
\left(\frac{-Bct+y}{\sqrt{1+B^2}}-\omega_Y^+ t\right)\nonumber\\
&=&e^{-\gamma t}e^{\frac{Bc\,y}{2(1+B^2)}}\Gamma_u^\theta\left(\frac{y}{\sqrt{1+B^2}}\right)\nonumber\\
&=&e^{-\gamma t}e^{\frac{Bc\,y}{2(1+B^2)}}\Gamma^\theta_{u,\infty}\left(0,\frac y\rac \right)\nonumber\\
&\geq&C_\theta \,e^{-\gamma
t}e^{\frac{Bc\,y}{2(1+B^2)}}e^{-2\sqrt{\frac{2\theta}{1+B^2}}|y|},\nonumber
\end{eqnarray}
in view of Lemma \ref{lem:modified_tails}. It follows from this and \eqref{est-tails} that the ordering
\begin{equation}\label{numero2}
n(t,ct,y)\leq C\,\bar
n(t,ct,y), \quad\text{ for all } t\geq 0, y\in \R,
\end{equation}
is guaranteed if
$\mu:=2\sqrt{\frac{2\theta}{1+B^2}}-\frac{Bc}{2(1+B^2)}>0$ (positivity is insured by $\theta\geq \frac{B^2c^2}{64(1+B^2)}$),
$\gamma:=\min(\bar\gamma,\frac{c^2}4+\lambda_\infty)>0$ (notice that $c>c^*$ implies  $\lambda_\infty+\frac{c^2}4>0$) and $C\geq \frac{C_\mu}{C_\theta}$.

Moreover, since $n_0(\cdot,\cdot)$ has a compact
support and $\bar n(0,\cdot,\cdot)>0$, we have
\begin{equation}\label{numero3}
n_0(x,y)\leq C\bar n(0,x,y), \quad \text{ for all } x\leq 0, y\in \R,
\end{equation}
if $C>0$ is large enough.

It follows from \eqref{eq-demi-espace}, \eqref{eq:super_ru},
\eqref{numero2}, \eqref{numero3} and the parabolic comparison
principle on $\{(t,x+ct,y);\,t\geq 0,\,x\leq 0,\,y\in\R\}$ that
for any $t\geq 0$, $x\leq 0$ and $y\in\R$,
\begin{eqnarray*}
n(t,ct+x,y)&\leq& C\,\bar n(t,ct+x,y)\\
&=&C e^{-\gamma t}e^{\frac{-\omega_Y^+}{2\sqrt{1+B^2}}(y-Bx)}\Gamma_{u,\infty}^\theta\left(x+ct+\frac{B}{\rac}\omega _Y^{+}t,y-\frac{\omega _Y ^{+}}{\rac}t\right),
\end{eqnarray*}
where we have used the expression \eqref{speeds-invasion} for $\omega _Y^+$. Combining again  \eqref{speeds-invasion} with \eqref{2}, we arrive at
\begin{eqnarray}
n(t,ct+x,y)
&\leq& C C_\theta e^{-\gamma t}e^{\frac{-\omega_Y^+}{2\sqrt{1+B^2}}(y-Bx)}e^{-\sqrt{
\frac{\theta}{2(1+B^2)}}\left|y-Bx\right|}\nonumber\\
&\leq& C C_\theta e^{-\gamma t}e^{-\frac 12 \sqrt{\frac \theta{2(1+B^2)}}\left|y-Bx\right|},\nonumber\label{eq-bornebarn}
\end{eqnarray}
using the fact that $\theta>2{\omega_Y^+}^2$.

The estimate \eqref{eq-bornebarn} for $x\leq 0$ and the estimate
 \eqref{est1} for $x\geq 0$ are enough to prove \eqref{point1}.
\end{proof}

\begin{proof}[Proof of $(ii)$]  The proof of \eqref{point2-minoration}, that
is of the survival of the population around $(t,ct,0)$ for $t\geq
0$ is similar to the proof of Theorem~\ref{th:survival}. It is
indeed possible to extend the proof of Theorem~\ref{th:survival}
to the present assumption on $r$, using Lemma~\ref{lem:tails2} to
estimate the tails of the density $n$. We skip the details of this
modification.

We now turn to the proof of the estimate
\eqref{point2-majoration}. Our approach will be similar to the
proof of Theorem~\ref{th:invasion}, more specifically, the proof
of estimates \eqref{par-dessus1}. For some $\theta>0$ to be
determined later, we seek a solution of
\begin{equation}\label{sol-linuA}
\partial _t \psi-\partial _{XX} \psi -\partial_{YY} \psi-\bar
r_u^\theta\left(\rac \,Y+Bct\right) \psi=0,
\end{equation}
in the form
\begin{equation*}
 \psi(t,X,Y):=e^{-\gamma t}e^{\omega X}e^{-\frac {\omega_Y^+}2(Y-\omega _Y^+
t)}\Gamma_u^\theta \left(Y-\omega _Y^+ t\right),
\end{equation*}
with $\gamma >0$, $\omega>0$ to be chosen. We see that $\psi$ is a solution of
\eqref{sol-linuA} if and only if
\begin{equation}\label{cond-sol-linuA}
-\gamma-\omega^2+\frac{{\omega
_Y^+}^2}4+\lambda_{u,\infty}^\theta=0.
\end{equation}

Next, we define
\begin{equation}\label{def:barn2}
\bar
n(t,x,y):=\psi\left(t,\frac{x+By}{\sqrt{1+B^2}},\frac{-Bx+y}{\sqrt{1+B^2}}\right),
\end{equation}
which is then a supersolution of \eqref{eq_n} as soon as \eqref{cond-sol-linuA} holds. As far as the line $x=ct$ is concerned, we have
\begin{eqnarray*}
 \bar n(t,ct,y)&=&e^{-\gamma t}
e^{\omega\left(\frac{ct+By}\rac\right)}
 e^{\frac{-\omega_Y^+}2\left(\frac{-Bct+y}{\sqrt{1+B^2}}-\omega_Y^+t\right)}\Gamma_u^\theta
\left(\frac{-Bct+y}{\sqrt{1+B^2}}-\omega_Y^+ t\right)\nonumber\\
&=&e^{\left(\frac{c\omega }{\rac}-\gamma \right)t+\left(\frac{\omega B-\omega_Y^+/2}\rac\right) y}\Gamma_{u,\infty}^\theta\left(0,\frac y{\rac}\right),\nonumber\\
\end{eqnarray*}
by using the definition \eqref{speeds-invasion} of $\omega _Y^{+}$.
In view of Lemma~\ref{lem:modified_tails}, this yields
\begin{equation}\label{ineqbarn2}
 \bar n(t,ct,y)\geq C_\theta e^{\left(\frac{c\omega }{\rac}-\gamma
\right)t}e^{-\left(2\sqrt{2\theta}+|\omega
B-\omega_Y^+/2|\right)\frac{|y|}{\sqrt{1+B^2}}}.
\end{equation}
We select $\gamma:=\frac{c\omega
}{\rac}$ and $\omega$ the positive solution of \eqref{cond-sol-linuA}.
There is a solution of the second order polynomial \eqref{cond-sol-linuA} provided $\theta>0$ is large enough, since its discriminant satisfies
$\Delta=\frac{c^2}{1+B^2}+\left({\omega_Y^+}^2+4\lambda_{u,\infty}^\theta\right)\to_{\theta\to\infty}\frac{c^2}{1+B^2}+\frac{B^2\left(c^2-{c_u^{**}}^2\right)}{1+B^2}>0$. We can also assume that $
\theta >8(\omega B+\omega _Y ^+/2)^2$ so that, in particular,
$2\sqrt{2\theta}>2|\omega B-\omega_Y^+/2|$, and then
thanks to Lemma~\ref{lem:tails2} and \eqref{ineqbarn2},
there exists a constant $C>0$ such that, for all $t\geq 0$ and
$y\in\R$,
\[n(t,ct,y)\leq C\bar n(t,ct,y).\]
The ordering at $t=0$ being obtained as in $(i)$ above, the
comparison principle implies that $n(t,x+ct,y)\leq C\bar
n(t,x+ct,y)=C
\psi\left(t,\frac{x+By}{\sqrt{1+B^2}},\frac{-Bx+y}{\sqrt{1+B^2}}\right)$,
for all $(t,x,y)\in \R_+\times\R_-\times\R$. Using the definition
of $\psi$, we arrive at
$$
n(t,x+ct,y)\leq Ce^{\omega\rac x}e^{\left(\frac{\omega B}{\rac}-\frac{\omega _Y ^{+}}{2\rac}\right)(y-Bx)}\Gamma _u ^{\theta}\left(\frac{-B(x+ct)+y}{\rac}-\omega _Y ^{+}t\right).
$$
Going back to $\Gamma _{u,\infty}^{\theta}$ and using estimate \eqref{2}, we end up with
\begin{eqnarray*}
n(t,x+ct,y)&\leq& CC_\theta e^{\omega\rac x}e^{\left(\frac{\omega B}{\rac}-\frac{\omega _Y ^{+}}{2\rac}\right)(y-Bx)} e^{-\sqrt{\frac{\theta}{2(1+B^2)  }}\vert y-Bx\vert}
\\
&\leq& CC_\theta e^{\omega\rac x} e^{-\frac 12\sqrt{\frac{\theta}{2(1+B^2)  }}\vert y-Bx\vert},
\end{eqnarray*}
using $\theta >2(2 \omega B+\omega _Y ^+)^2$. This estimates
proves \eqref{point2-majoration}.
\end{proof}

\begin{proof}[Proof of $(iii)$]
 The proof of Proposition \ref{prop:extinction} shows that \eqref{controle1} still holds true, and then, thanks
  to Lemma \ref{lem:tail_Gamma_m}, for any $\mu>0$, there exists $C_\mu>0$ such that, for any $(x,y)\in\R^2$,
\begin{eqnarray*}
e^{\frac{cx}2}n(t,x+ct,y)&\leq& C_\mu e^{\left(-\lambda_\infty-\frac{c^2}4\right)t}\Gamma_\infty(x,y)\nonumber\\
&\leq&
C_\mu e^{\left(-\lambda_\infty-\frac{c^2}4\right)t}e^{-\mu\max\left(|y-Bx|,x\right)},
\end{eqnarray*}
which is enough to prove \eqref{point3-majoration}.

We now turn to the proof of the survival of the population, with a
shift slower than the climate change, that is estimate \eqref{point3-minoration}. For ease of writing we take $K\equiv 1$, which
is harmless since $0<k^-\leq K \leq k^+$. The proof shares some arguments with that of \eqref{minoration-n}. First,
for a given $(t_0,x_0,y_0)\in[ 1,\infty)\times \R^2$, we need a control of the nonlocal term $\int _\R n(t_0,x_0,y')dy'$. We
claim that we can reproduce the arguments used to prove \eqref{est-nonlocal}. Indeed, the crucial control of the tails \eqref{tails}
in the unconfined case is replaced by \eqref{tails-mixed} in the present mixed scenario. Hence we can reproduce the proof
of subsection \ref{ss:invasion-unconfined}: we derive \eqref{bidule}, and apply Theorem~\ref{th:harnack-mod} to obtain \eqref{est-nonlocal}. As a result, for given $\mu >0$, there is $C>0$ (as in Lemma \ref{lem:tails2})  that ,
 for a given $(t_0,x_0,y_0)\in[1,\infty)\times \R^2$ and $M>2$ such that $\vert y_0-B(x_0-ct_0)\vert \leq M$, there is $C_M<\infty$, such that
\begin{equation}\label{estim-har2}
\int_\R n(t_0,x_0,y')dy'\leq C_M n(t_0,x_0,y_0)+\frac{3C}{\mu}e^{-\mu M}.
\end{equation}

Next, for $R>0$, proceeding as in the proof of Theorem
\ref{th:invasion} --- that is multiplying $\tilde \Gamma
_{u,R}(Y):=\Gamma_{u,R}^{1D}(\rac \, Y)$, which solves an
analogous of \eqref{pb-vp-Y}, by $\cos\left( \frac x R \frac \pi
2\right)$--- one can construct $\Gamma _{u,R}(X,Y)$ which solves
an analogous of \eqref{pb-vp-XY}, namely
\begin{equation}\label{pb-vp-XYbis}
\begin{cases}
-\partial_{XX} \Gamma_{u,R}-\partial_{YY}\Gamma_{u,R}-\bar
r _u\left(\rac\, Y\right)\Gamma_{u,R}=\lambda_{u,R}\Gamma_{u,R}  \quad\text{ in  }(-R,R)^2\vspace{3pt}\\
\Gamma _{u,R}=0 \quad\text{ on }  \partial ((-R,R)^2)\vspace{3pt}\\
\Gamma_{u,R} >0\quad \text{ on }  (-R,R)^2,\quad \Gamma
_{u,R}(0,0)=1,
 \end{cases}
\end{equation}
where $\lambda _{u,R}\to \lambda _{u,\infty}$ as $R\to\infty$.

For some $R>0$ and $\beta$ to be chosen later,
the function
\[\varphi(t,x,y):=\beta e^{-\frac {Bc}{2\rac}\left(\frac{Bx-y}\rac-\frac {Bc}\rac t\right)}\Gamma_{u,R}\left(x-\frac{B^2ct}{1+B^2},y+
\frac{Bct}{1+B^2}\right)\] satisfies $\varphi(t,\cdot,\cdot)=0$ on
$\partial\left(\left(\frac{B^2ct}{1+B^2},-\frac{Bct}{1+B^2}\right)+(-R,R)^2\right)$,
and for
$(x,y)\in\left(\left(\frac{B^2ct}{1+B^2},-\frac{Bct}{1+B^2}\right)+(-R,R)^2\right)$,
some straightforward computations yield
\begin{align}
&\partial_t\varphi(t,x,y)-\Delta
\varphi(t,x,y)-r_u\left(x-\frac{B^2ct}{1+B^2},y+
\frac{Bct}{1+B^2}\right)\varphi(t,x,y)\nonumber\\
&=\left(\frac{B^2c^2}{4(1+B^2)}+\lambda_{u,R}\right)\varphi(t,x,y)\nonumber
\\
&=\left(\lambda_{u,R}-\lambda_{u,\infty}+\frac{B^2(c^2-(c_u^{**})^2)}{4(1+B^2)}\right)\varphi(t,x,y)\nonumber\\
&\leq
\frac{B^2(c^2-(c_u^{**})^2)}{8(1+B^2)}\varphi(t,x,y),\label{estim}
\end{align}
if we fix $R>1$ large enough since $c^2-(c_u^{**})^2<0$ and
$\lambda _{u,R}\to \lambda _{u,\infty}$ as $R\to\infty$. Now, observe that $r_u\left(x-\frac{B^2ct}{1+B^2},y+
\frac{Bct}{1+B^2}\right)=r_u(x-ct,y)$; also there is $T>0$ sufficiently large so that, for all $t\geq T$ and all
$(x,y)\in\left(\left(\frac{B^2ct}{1+B^2},-\frac{Bct}{1+B^2}\right)+(-R,R)^2\right)$, $x-ct\leq 0$ so that
$r_u(x-ct,y)=r(x-ct,y)$. As a result \eqref{estim} is recast as
\begin{equation}\label{estim2}
\partial_t\varphi(t,x,y)-\Delta
\varphi(t,x,y)-r\left(x-ct,y\right)\varphi(t,x,y)\leq  \frac{B^2(c^2-(c_u^{**})^2)}{8(1+B^2)}\varphi(t,x,y),
\end{equation}
for all $t\geq T$, all $(x,y)\in \left(\left(\frac{B^2ct}{1+B^2},-\frac{Bct}{1+B^2}\right)+(-R,R)^2\right)$.

We can assume that $\beta>0$ is small enough so that $
 \varphi(T,\cdot,\cdot)<n(T,\cdot,\cdot)$. Assume by contradiction that the set $\{t\geq T:\,\exists
(x,y),\,n(t,x,y)=\varphi(t,x,y)\}$ is non empty, and define
$$
t_0:=\min\{t\geq T:\,\exists
(x,y),\,v(t,x,y)=\varphi(t,x,y)\}\in (T,\infty).
$$
Hence, $\varphi -u$ has a zero maximum value at some point
$(t_0,x_0,y_0)$. This implies that
$$
\left[\partial_t(\varphi-n)-\Delta(\varphi -u)-r(x_0-ct_0,y_0)(\varphi-u)\right](t_0,x_0,y_0)\geq 0.
$$
 Combining \eqref{estim2} with \eqref{estim-har2}, we get
  $$
  0\leq \frac{B^2(c^2-(c_u^{**})^2)}{8(1+B^2)}+C_M\varphi (t_0,x_0,y_0)+\frac{2C}{\mu}e^{-\mu M}.
  $$
Selecting successively $M>2$ large enough and $\beta >0$ small enough,
we get $0\leq c^{2}-(c_u^{**})^2$, that is a contradiction. As a result, we have
\begin{equation}\label{pardessous}
n(t,x,y)\geq \beta e^{-\frac {Bc}{2\rac}\left(\frac{Bx-y}\rac-\frac {Bc}\rac t\right)}\Gamma_{u,R}\left(x-\frac{B^2ct}{1+B^2},y+
\frac{Bct}{1+B^2}\right),
\end{equation}
for all $t\geq T$, all $(x,y)\in \left(\left(\frac{B^2ct}{1+B^2},-\frac{Bct}{1+B^2}\right)+(-R,R)^2\right)$. Now, for a given $-1\leq x_0\leq 1$, the above yields
\begin{eqnarray*}
\int _\R n\left(t,x_0+\frac{B^{2}ct}{1+B^{2}},y\right)\,dy&\geq& \beta e^{-\frac{B^{2}c x_0}{\rac}}\int _\R e^{\frac{Bc}{1+B^{2}}\left(y+\frac{Bct}{\rac}\right)}\Gamma _{u,R}\left(x_0,y+\frac{Bct}{\rac}\right)\,dy \nonumber\\
&\geq&  \beta e^{-\frac{B^{2}c }{\rac}}Ó\min _{-1\leq x\leq 1}\int _\R e^{\frac{Bc}{1+B^{2}}z}\Gamma _{u,R}\left(x,z\right)\,dz,
\end{eqnarray*}
which concludes the proof of \eqref{point3-minoration}.
\end{proof}

\begin{proof}[Proof of $(iv)$] Let $R>0$ to be chosen later. Since $c<c_u^{**}$, we can follow the above proof of \eqref{point3-minoration}
and get \eqref{pardessous}, which in turn provides a small enough $\eta >0$ such that, for all $t\geq T$, all $\max(|x|,|y|)\leq R$,
\begin{equation}\label{machintruc}
n\left(t,x+\frac{B^2ct}{1+B^2},y-\frac{Bct}{1+B^2}\right)\geq \eta.
\end{equation}
Also, since $c<c^{*}$, we can follow the proof of \eqref{point2-minoration} (see also Theorem \ref{th:survival}) and get, up to reducing $\eta >0$, that for all $t\geq T$, all $\max(|x|,|y|)\leq R$,
\begin{equation}\label{machin2}
n\left(t,x+ct,y\right)\geq \eta.
\end{equation}

For $R>0$, $\tilde \Gamma
_{u,R}(Y):=\Gamma_{u,R}^{1D}(\rac \, Y)$ solves
\begin{equation}\label{pb-vp-Ydeux}
\begin{cases}
-\partial_{YY}\tilde \Gamma_{u,R}-\bar
r_u\left(\rac\, Y\right)\tilde \Gamma_{u,R}=\tilde \lambda_{u,R}\tilde \Gamma_{u,R}  \quad\text{ in  }(-R,R)\vspace{3pt}\\
\tilde \Gamma _{u,R}=0 \quad\text{ on }  \partial ((-R,R))\vspace{3pt}\\
\tilde \Gamma_{u,R} >0\quad \text{ on }  (-R,R),\quad \tilde \Gamma _{u,R}(0)=1,
 \end{cases}
 \end{equation}
 with $\tilde \lambda_{u,R} \to
\lambda_{u,\infty}$ as $ R\to\infty$.

Define
\begin{eqnarray*}
\bar n(t,x,y):&=&\beta e^{-\frac{\omega_Y^+}2\left(\frac{-Bx+y}{\sqrt{1+B^2}}-\omega_Y^+t\right)}\tilde \Gamma_{u,R}\left(\frac{-Bx+y}{\sqrt{1+B^2}}-\omega_Y^+t\right)\\
&=&\beta e^{-\frac{\omega_Y^+}2\left(\frac{y-B(x-ct)}{\sqrt{1+B^2}}\right)}\tilde \Gamma_{u,R}\left(\frac{y-B(x-ct)}{\sqrt{1+B^2}}\right),
\end{eqnarray*}
with $\beta >0$ to be chosen later. We aim at applying the comparison principle on the domain $D:=\{(t,x,y):\,t\geq T, \frac{B^2ct}{1+B^2}\leq x\leq ct, \vert y-B(x-ct)\vert \leq R\}$.

We have
\begin{eqnarray*}
\partial_t\bar n(t,x,y)-\Delta \bar n(t,x,y)-\bar r_u(y-B(x-ct))\bar n(t,x,y)
&=&\left(\frac{{\omega_Y^+}^2}4+\tilde \lambda_{u,R}\right)
\bar n(t,x,y)\\
&=&\left(\frac{B^2(c^2-{c_u^{**}}^2)}{1+B^2}+\tilde\lambda_{u,R}-\lambda_{u,\infty}\right)\bar n(t,x,y)\\
&\leq& \frac{B^2(c^2-{c_u^{**}}^2)}{1+B^2}\bar n(t,x,y) ,
\end{eqnarray*} provided $R>0$ is large enough.

Concerning the boundary of $D$, if $\vert y-B(x-ct)\vert =R$ then $\bar n(t,x,y)=0\leq n(t,x,y)$. If $x=ct$ then
$$
\bar n(t,ct,y)\leq \beta e^{\frac{\omega_Y^{+}}{2}\frac{R}{\rac}}\Vert \tilde \Gamma_{u,R}\Vert _\infty \leq \eta \leq n(t,ct,y),
$$
provided $\beta >0$ is small enough. If $x=\frac{B^{2}ct}{1+B^2}$, then
$$
\bar n\left(t,\frac{B^2ct}{1+B^2},y\right)\leq \beta e^{\frac{\omega_Y^{+}}{2}\frac{R}{\rac}}\Vert \tilde \Gamma_{u,R}\Vert _\infty \leq \eta \leq n\left(t,\frac{B^2ct}{1+B^2},y\right),
$$
provided $\beta >0$ is small enough.
If $t=T$, $n(T,x,y)>0$, an we then have $\bar
n(T,x,y)\leq n(T,x,y)$  for all
$\frac{B^2cT}{1+B^2}\leq x\leq cT$, $\vert y-B(x-cT)\vert \leq R$, provided $\beta >0$ is small enough.

By using again (details are omitted)  Theorem~\ref{th:harnack-mod}, as done in subsection \ref{ss:invasion-unconfined} and in the proof of $(iii)$ above, we deduce that
$$
n(t,x,y)\geq \bar n(t,x,y)=\beta e^{-\frac{\omega_Y^+}2\left(\frac{y-B(x-ct)}{\sqrt{1+B^2}}\right)}\tilde \Gamma_{u,R}\left(\frac{y-B(x-ct)}{\sqrt{1+B^2}}\right),
$$
for all $(t,x,y)$ such that $t\geq T$, $\frac{B^2ct}{1+B^2}\leq x\leq ct$, $\vert y-B(x-ct)\vert \leq R$. This is enough to prove \eqref{point3-minoration}.
\end{proof}

\section*{Acknowledgements}
The second author acknowledges support from the European Research Council under the European Union's Seventh Framework Programme (FP/2007-2013) /ERC Grant Agreement n. 321186 - ReaDi - "Reaction-Diffusion Equations, Propagation
and Modelling" held by Henri Berestycki. The third author acknowledges support from the ANR under grant  Kibord: ANR-13-BS01-0004, and MODEVOL: ANR-13-JS01-0009.

\signma
\signhb
\signgr

\end{document}